\documentclass[12pt]{article}
\usepackage{amssymb}
\usepackage{amssymb,amsmath,latexsym}

\usepackage[dvips]{color}
\allowdisplaybreaks

\numberwithin{equation}{section}

\newtheorem{thm}{Theorem}[section]
\newtheorem{proposition}[thm]{Proposition}
\newtheorem{remark}[thm]{Remark}
\newtheorem{definition}[thm]{Definition}

\newtheorem{lemma}[thm]{Lemma}
\newtheorem{assumption}[thm]{Assumption}

\def\beq{\begin{equation}}
\def\eeq{\end{equation}}

\def\<{\langle}  \def\>{\rangle}

\newcommand{\R}{\mathbb{R}}
\newcommand{\bP}{\mathbb{P}}
\newcommand{\bE}{\mathbb{E}}

\newcommand{\cF}{{\cal F}}
\newcommand{\cL}{{\cal L}}

\begin{document}

\def\qed{{\hfill $\Box$ \bigskip}}

\title{\Large \bf $L\log L$ criterion for a class of  multitype superdiffusions with non-local branching mechanisms}

\author{{\bf Zhen-Qing Chen}
\quad {\bf Yan-Xia Ren}\thanks{The research of this author is supported by NNSFC (Grant No.  11731009 and 11671017).} \quad  \hbox{and} \quad {\bf Renming Song
}\thanks{Research supported in part by a grant from the Simons
Foundation (\#429343, Renming Song).}}

\date{}
\maketitle

\begin{abstract}
In this paper, we provide a pathwise spine decomposition for multitype superdiffusions with non-local branching mechanisms under a martingale change of measure.  As an application of this decomposition, we obtain
a necessary and sufficient condition (called the $L\log L$ criterion)
for the limit of the fundamental martingale to be non-degenerate.  This result complements the related results obtained in \cite{KLMR,KM,LRS09} for superprocesses with purely local branching mechanisms and in \cite{KP} for
super Markov chains.
\end{abstract}


\medskip


\begin{doublespace}

\section{Introduction}

\subsection{Previous results}

Suppose that $\{Z_n, n\ge 1\}$ is a Galton-Watson  process
with offspring distribution $\{p_n\}$.
That is, each particle lives for one unit of time; at the time of its death,
it gives birth to $k$ particles with probability $p_k$ for $k=0, 1, \cdots$;
and $Z_n$ is the total number of particles alive at time $n$.
Let $L$ be a random variable with distribution $\{p_n\}$
and   $m:=\sum^{\infty}_{n=1}np_n$ be the
expected number of
offspring per particle. Then $Z_n/m^n$ is a non-negative martingale.
Let $W$ be the limit of $Z_n/m^n$ as $n\to\infty$. Kesten and Stigum
proved in \cite{KS} that, when $1<m<\infty$ (that is, in the supercritical case),
$W$ is non-degenerate (i.\/e., not almost
surely zero) if and only if
\begin{equation}\label{LLogL-GW}
E(L\log ^+L)=\sum^{\infty}_{n=1}p_nn\log n<\infty.
\end{equation}
This result is usually called  the Kesten-Stigum $L\log L$
criterion. In \cite{AH76a}, Asmussen and Herring generalized this
result to the case of branching Markov processes under some
conditions.

In 1995, Lyons, Pemantle and Peres developed a martingale change of
measure method in \cite{LPP} to give a new proof for the
Kesten-Stigum $L\log L$ criterion for (single type) branching processes.
Later this approach was applied
to prove the $L\log L$ criterion for
multitype and general multitype branching processes in \cite{BK, KLPP}.

In \cite{LRS09}, the martingale change of measure method was used to prove
an $L\log L$ criterion for a class of superdiffusions.
In this paper, we will establish a pathwise spine decomposition for multitype
superdiffusions with purely non-local branching mechanisms.
Our non-local branching mechanisms are special in the sense
that the types of the offspring are different from their mother,
but their spatial locations at birth are the same as their mother's spatial location immediately before her death.
We will see below that, a multitype superdiffusion with
a purely non-local branching mechanism given by \eqref{psi} below can also be
viewed as a superprocess having
a switched diffusion  as its spatial motion and $\widehat \psi(x,i; \cdot)$ defined in \eqref{hat-psi}
as its (non-local) branching mechanism.
 Using
a non-local Feynman-Kac transform, we prove that,
under a martingale change of measure, the spine runs as a copy of
an $h$-transformed switched-diffusion, which is a new switched diffusion.
The non-local nature of the branching mechanism induces a different
kind of immigration--the \textit{switching-caused} immigration.
That is to say, whenever there is a switching of types, new immigration happens
and the newly immigrated particles choose their types according to
a distribution $\pi$.
The switching-caused immigration is a consequence of the non-local branching,
and it does not occur when the branching mechanism is purely local.
Note that in this paper we do not consider branching mechanism with a local term.
Note also that our non-local branching mechanism is special in the sense that
only the types, not the spatial  locations at birth, are different from the mother's.
It is interesting to consider superprocesses with more general
non-local branching mechanism and with local branching mechanism.
For this case, one can see the recent preprint \cite{RSY}, where the spine is a concatenation process.

Concurrently to our work, Kyprianou and  Palau \cite{KP} considered super
Markov chains with local and non-local branching mechanisms. Note that if
particles do not move in space, our model reduces to the model considered
in  \cite{KP} with purely non-local branching mechanism.
Kyprianou and  Palau  \cite{KP} also found that immigration happens when
particle jumps (they call this immigration {\it jump immigration}), which
corresponds to our switching-caused immigration.

\subsection {Model: multitype superdiffusions}

For integer $K\geq 2$, a $K$-type
superdiffusion is defined as follows. Let
$S:=\{1,2,\cdots,K\}$ be the set of types.
For each $k\in S$, ${\cL}_k$ is a second order
elliptic differential operator of divergence form:
\begin{equation}\label{def-L}{\cL}_k=\sum^d_{i,j=1}\frac{\partial}{\partial x_i}\left(a^{(k)}_{i,j}\frac{\partial}{\partial
x_j}\right)\quad\mbox{ on }\R^d, \end{equation} with
$A^k(x)=(a^k_{ij}(x))_{1\leq i, j\leq d}$ being a symmetric
 matrix-valued function on $\R^d$ that is uniformly elliptic and
 bounded:
$$ \Lambda_1|v|^2\le\sum^d_{i,j=1}a^k_{i,j}(x)v_iv_j\le \Lambda_2|v|^2
\qquad\mbox{ for all } v\in\R^d\mbox{ and }x\in \R^d
 $$ for some
positive constants $0<\Lambda_1\leq \Lambda_2<\infty$, where
$a^k_{ij}(x)\in C^{2,\gamma}(\R^d), 1\leq i, j\leq d$  for some
$\gamma \in (0, 1)$.
Throughout this paper,  for $i=1, 2, \cdots$, $ C^{i,\gamma}(\R^d)$ stands for the space of $i$ times continuously differentiable functions with all their $i$th order derivatives belonging to $C^\gamma(\R^d)$,
the space of $\gamma$-H\"older continuous functions on $\R^d$.

 Suppose that for each $i\in S$,
$\xi^i :=\{\xi^i_t, t\ge 0; \Pi^i_x, x\in \R^d\}$ is a diffusion process on $\R^d$  with generator ${\cL}_i$.
In this paper we will always assume that
 $D$ is a domain  of finite Lebesgue measure in $\mathbb{R}^d$.
 For $x\in D$, denote by $\xi^{i, D}:= \{ \xi^{i, D}_t, t\ge 0; \Pi^i_x, x\in D\}$ the subprocess of $\xi^i$
killed upon exiting $D$;
 that is,
$$
\xi^{i,D}_t=\left\{\begin{array}{ll} \xi^i_t & \mbox{ if } t<\tau^i_D,\cr
\partial, & \mbox{ if } t\ge \tau^i_D,\cr
\end{array}
\right.
$$
where $\tau^i_D =\inf\{t\ge0; \xi^{i}_t\notin D\}$ is the first exit time of
$D$ and $\partial$ is a cemetery point.

Let ${\cal M}_1(S)$
denote the set of all  probability measures on $S$,  and ${\cal M}_F(\mathbb{R}^d\times S)$ denote the space of finite measures on $D\times S$. For any measurable set $E$,
we use $B_b(E)$ (resp. $B^+_b(E)$) the family of
bounded (resp. bounded positive)  ${\cal B}(E)$-measurable functions on $E$.
Any function $f$ on $D$ is
automatically extended to
$D_\partial := D\cup \{\partial\}$ by setting
$f(\partial)=0$. Similarly, any function $f$ on $D\times S$ is
automatically extended to
$D_\partial \times S$ by setting
$f(\partial, i)=0, i\in S$.
If $f(t, x, i)$ is a function on $[0,+\infty)\times D\times S$, we say $f$ is \textit{locally bounded} if $\sup_{t\in [0,T]}\sup_{(x, i)\in D\times S}|f(t, x, i)|<+\infty$ for every finite $T>0$. For a function $f(s, x, i)$ defined on $[0,+\infty)\times D\times S$ and a number $t\ge 0$, we denote by $f_t(\cdot)$ the function $(x, i)\mapsto f(t, x, i)$.
For convenience we use the following convention
throughout this paper: For any probability measure $\bP$, we also use
$\bP$ to denote the expectation with respect to $\bP$. When there is
only one probability measure involved, we sometimes also use $\bE$ to
denote the expectation with respect to that measure.

We consider a multitype superdiffusion $\{\chi_t, t\ge 0\}$ on $D$, which is
a strong Markov process taking values in
${\cal M}_F(D\times S)$.
We can represent ${\chi}_t$ by $(\chi^1_t,\cdots, \chi^K_t)$ with $\chi^i_t \in {\cal M}_F (D)$
for  $1\leq i\leq K$.
For $f\in{B}^+_b(D\times S)$,
we often use the convention
$$
f(x)=(f(x,1),\cdots, f(x,K))=(f_1(x),\cdots, f_K(x)),
\quad x\in D,
$$
and $\langle f, \chi_t\rangle=\sum^K_{j=1}\langle f_j, {\chi}^j_t\rangle$ .
Suppose that $F(x, i; du)$ is a kernel from $D\times S$ to $(0,\infty)$ such
that, for each $i\in S$, the function
$$
m(x,i):=\int^\infty_0u F(x, i; du)
$$
is bounded on $D$.
Let $n$ be a bounded Borel function on $D\times S$
such that $n(x,i)\ge m(x, i)$ for every $(x, i)\in D\times S$, and
 $p^{(i)}_j(x)$, $i, j\in S$, be non-negative Borel functions on $D$ with
$\sum^K_{j=1}p^{(i)}_j(x)=1$. Define
$$
\pi(x,i;\cdot)=\sum^K_{j=1}p^{(i)}_j(x)\delta_{(x,j)}(\cdot),
$$
where $\delta_{(x,j)}$ denotes the unit mass at $(x,j)$.
Then $\pi(x,i;\cdot)$ is a Markov kernel on $D\times S$.
For any $f\in{B}^+_b(D\times S)$, we write
$\pi(x,i;f)=\sum^K_{j=1}p^{(i)}_j(x)f_j(x).$
Define
$$
\zeta(x,i;f)=n(x,i)\pi(x,i;f)+
\int^\infty_0\left(1-e^{-u\pi(x,i;f)}-u\pi(x,i;f)\right)F(x,i;du).
$$
Note that we can rewrite  $\zeta(x,i;f)$ as
 $$
\zeta(x,i;f)=\widetilde n(x,i)\pi(x,i;f)+
\int^\infty_0\left(1-e^{-u\pi(x,i;f)}\right)F(x,i;du),
$$
where
\begin{equation}\label{def-tilde-d}
\widetilde n(x,i):=n(x,i)-m(x, i)\ge 0.
\end{equation}
 $\zeta(x,k;f)$ serves as the non-local branching mechanism, which is a
 special form of \cite[(3.17)]{DGL}  with $d$ (corresponding to $n$ in the present paper) and $n$ (corresponding to
 $F$ in the present paper) independent of $\pi$, and $G(x, i; d\pi)$ being the unit mass at some
 $\pi(x,i;\cdot)\in {\cal M}_1(S)$,
 that is, the non-locally displaced offspring born at $(x,i)\in D\times S$ choose their types independently
according to the (non-random) distribution $\pi(x,i; \cdot )$.
Suppose $b(x,i)\in B^+_b(D\times S)$.
Put
 \begin{equation}\label{psi}
 \psi(x,i;f)=b(x,i)\left(f_i(x)-\zeta(x,i;f)\right),   \quad (x,i)\in D\times S, \
 f\in B^+_b(D\times S).
\end{equation}
 Without loss of generality, we suppose that $p^i_i(x)=0$ for
 all $(x,i)\in D\times S$, which means that $\psi$ is a purely
 non-local branching mechanism.
The Laplace-functional of $\chi$ is given by
\begin{equation}\label{Laplace}
\bP_{\mu}\exp\langle -f, {\chi}_t\rangle =\exp\langle-
u^f_t(\cdot),\mu\rangle,
\end{equation}
where $u^f_t(x, i)$ is the unique locally bounded positive solution to
the evolution equation
\begin{equation}\label{int}\begin{array}{rl}
u^f_t(x,i)+&\displaystyle\Pi^i_{x}\left[
\int^t_0\psi(
\xi^{i,D}_s, i; u^f_{t-s})ds\right]=\Pi^i_xf_i(\xi^{i, D}_t), \quad\mbox{for }
t\ge 0,\end{array}
\end{equation}
where we used the convention that $u^f_{t}(x)=(u^f_{t}(x, 1), \cdots,
u^f_{t}(x, K))$. This process is called an
$(({\cL}_1,\cdots {\cL}_K),   {\bf \psi})$-multitype superdiffusion in $D$.
 It is well known (see, e.g., \cite{F}) that for any non-negative bounded function $f$ on $D\times S$,
 the $u^f_t(x, i)$ in \eqref{int} is a locally bounded positive solution
 to  the following
 system of partial differential equations:
   for each $i\in S$,
 \begin{equation}\label{pde}
 \left\{
\begin{array}{rlll}&\displaystyle\frac{\partial{u^f(t,x, i)}}{\partial{t}}
& =& {\cL}_i (t,x, i) - \psi(x, i; u^f_t )
\quad (t, x)\in (0, \infty)\times D\\
&\displaystyle u^f(0, x, i)&=& f_i(x),
 \quad x\in D\\
& \displaystyle u^f(t, x, i)&=& 0 \quad
 (t, x)\in (0, \infty)\times \partial D.
\end{array}\right.
\end{equation}

Multitype superdiffusions
can be obtained as
a scaling limit of a  sequence of multitype branching diffusions.
See \cite{DGL} for details. The multitype superdiffusion $\chi$
considered in this paper are the scaling limits of multitype branching diffusions
whose types can change only at branching times.

Define
\begin{equation}\label{def-m}
r_{il}(x)=n(x,i)p^{(i)}_l(x)\quad  x\in D,\, i,l\in S.
 \end{equation}
Let $v(t, x,i)=\bP_{\delta_{(x,i)}}\langle f, \chi_t\rangle$.
Using \eqref{Laplace} and \eqref{int}, we see that for  all $(t, x, i)\in (0, \infty)\times D\times S$,
\begin{equation}\label{mean-int}
\begin{array}{rl}
v_t(x,i)=&\displaystyle\Pi^i_xf_i(\xi^{i, D}_t)+\Pi^i_{x}
\int^t_0b(\xi_s, i)
\left(\sum^K_{l=1}r_{il}(\xi^{i,D}_s)
v_{t-s}(\xi^{i,D}_s,l)-v_{t-s}(\xi^{i,D}_s, i)\right)ds.\end{array}
\end{equation}
 Then $v(t,x,i)$ is the unique locally bounded solution
to the following linear system (see, e.g.,  \cite{F}):
for each $i\in S$,
\begin{equation}\label{pde-mean}
\left\{
\begin{array}{rlll}&\displaystyle\frac{\partial{v(t,x, i)}}{\partial{t}}&=&
{\cL}_iv(t,x, i)+b(x, i)\sum^K_{l=1}(r_{il}(x)-\delta_{il})v(t,x,l),
 \quad (t, x)\in (0, \infty)\times D\\
&\displaystyle v(0, x, i)&=&f_i(x),
 \quad x\in D\\
&\displaystyle v(t, x, i)&=& 0, \quad
(t, x)\in (0, \infty)\times \partial D.
 \end{array}\right.
\end{equation}
Letting ${\bf v}(t, x)=(v(t,x, 1),\cdots, v(t,x, K))^T,$  we can
rewrite the partial differential  equations in \eqref{pde-mean} as
\begin{equation}\label{pde-mean2}
\frac{\partial}{\partial t}
{\bf v}(t,x)={\cal L}{\bf v}(t,x)+
B(x) \cdot ( R(x)-I){\bf v}(t,x),
\end{equation}
where
$${\cal L}=\left(\begin{array}{llll}{\cL}_1\quad &0&\cdots&\, 0\\
0\quad &{\cL}_2\,&\cdots&\, 0\\
\vdots\quad&\vdots\,&\ddots&\, \vdots\\
0\quad &0&\cdots&\, {\cL}_K\end{array}\right),
$$
$$
B(x)=\left(\begin{array}{llll}b(x, 1)\quad &0&\cdots&\,0\\
0\quad &b(x, 2)\,&\cdots&\, 0\\
\vdots\quad&\vdots\,&\ddots&\, \vdots\\
0\quad &0&\cdots&\, b(x, K)\end{array}\right)
$$
and
$$
R(x)=\left(\begin{array}{llll}r_{11}(x)\quad &r_{12}(x)&\cdots&\, r_{1d}(x)\\
r_{21}(x)\quad &r_{22}(x)\,&\cdots&\,r_{2d}(x)\\
\vdots\quad&\vdots\,&\ddots&\, \vdots\\
r_{K1}(x)\quad &r_{K2}(x)&\cdots&\, r_{KK}(x)\end{array}\right).
$$
In this paper we assume that  $B(x)\cdot R(x)$ is symmetric,
that is to say,
\begin{equation}\label{symm}
b(x,i)n(x,i)p^{(i)}_j(x)=b(x,j)n(x,j)p^{(j)}_i(x),\quad\mbox{ for all } i,j\in S, x\in D.\end{equation}
We assumed the symmetry of $B(x)\cdot R(x)$ and the symmetry of the operators ${\cL}_k$ for simplicity. If
the ${\cL}_k$'s are of
 non-divergence form and $B(x)\cdot R(x)$ is not symmetric,
 we can use the intrinsic ultracontractivity introduced in \cite{KS1}.

Note that
\begin{equation}\label{decom-M}
 R(x)-I=R(x)-N(x)+(N(x)-I),
\end{equation}
 where
$$
N(x)=\mbox{diag}\left(n(x,1),\cdots n(x,K)\right),\quad x\in D.
$$
Then by \eqref{def-m} and \eqref{decom-M},
\begin{equation}\label{decom-M'}
B(x)\cdot (R(x)-I)=\widehat B(x)\cdot \left( P(x)-I\right)
+B(x)\left(N(x)-I\right),\end{equation}
where
$$
\widehat B(x)
=\mbox{diag}\left(b(x, 1)n(x,1),\cdots, b(x, K)n(x,K) \right),
$$
and
$$
P(x)=\left( p_{ij}(x)\right)_{{i,j}\in S},\quad p_{ij}(x)=p^{(i)}_j(x).
$$
Put $Q(x)=(q_{ij}(x))_{{i,j}\in S}=\widehat B(x)\cdot (P(x)-I)$. We will assume that the matrix $Q$
is irreducible
on $D$ in the sense
that for any two distinct $k, l\in S$,
there exist $k_0, k_1, \cdots, k_r\in S$ with $k_i\neq k_{i+1}$, $k_0=k$, $k_r=l$
 such that $\{x\in D: q_{k_ik_{i+1}}(x)>0\}$ has positive Lebesgue
measure for each $0\leq i \leq  r-1$.
Let $\{(X_t, Y_t), t\ge 0\}$ be a switched diffusion
with generator ${\cal A}:= \cL + Q(x)$  killed upon exiting from $D\times S$,
and  $\Pi_{(x,i)}$  be its law starting from $(x,i)$.
$\{(X_t, Y_t), t\ge 0\}$ is a symmetric Markov process on
$D\times S$ with respect to $dx\times di$, the product of the
Lebesgue measure on $D$  and the counting measure on $S$.

Define
\begin{equation}\label{def-zeta1}
\zeta_1(x,i;f)=n(x,i)\pi(x,i;f)=n(x,i)\sum^K_{i=1}p^{(i)}_l(x)f_l(x)=
\sum^K_{l=1}r_{il}(x)f_l(x)
\end{equation}
and
\begin{equation}\label{def-zeta2}\zeta_2(x,i;f)=\int^\infty_0\left(1-e^{-u\pi(x,i;f)}-u\,\pi(x,i; f)\right)F(x,i;du).\end{equation}  Then \begin{equation}\label{zeta-1+2}\zeta(x,i;f)=\zeta_1(x,i;f)+\zeta_2(x,i;f).\end{equation}
Letting
$$
{u}^f(t,
x)=(u^f(t,x, 1),\cdots, u^f(t,x, K))^T \mbox{ and }
{\zeta}_2(x,f)=(\zeta_2(x,1; f),\cdots, \zeta_2(x,K;f))^T,
$$
in view of \eqref{psi}  we can
rewrite the partial differential
equation in \eqref{pde} as
\begin{equation}\label{pde2}
\frac{\partial}{\partial t}
{ u}^f(t,x)={\cal L}{u}^f(t,x)+
B(x) \cdot \left( R(x)-I\right){u}^f(t,x)+B(x) \cdot{\zeta}_2(x, u^f_t),
\end{equation}
which, by \eqref{decom-M}, is equivalent to
\begin{align}\label{pde3}
\frac{\partial}{\partial t}{ u}^f(t,x)&=
\ {\cal L}{u}^f(t,x)+\widehat B(x)\cdot
\left(P(x)-I\right){u}^f(t,x)\nonumber\\
&\quad +\ B(x)\cdot\left[(N(x)-I){u}^f(t,x)+{\zeta}_2(x, u^f_t)\right].
\end{align}
For $ f\in B^+_b(\R^d\times S)$, define
 \begin{equation}\label{hat-psi}
 \widehat \psi(x,i;f):=-b(x, i)n(x,i)f_i(x)+b(x, i)(f_i(x)-\zeta_2(x,i;f)).
\end{equation}
Then applying the strong Markov property of the switched diffusion process $(X, Y)$
at its first switching time and using the approach from \cite{CZ} (see in particular
p.296, Proposition 2.2 and Theorem 2.5 there) and \cite{F},
 one can verify using  \eqref{pde3}
that $u^f_t(x, i)$ satisfies
 \begin{equation}\label{int-equi}\begin{array}{rl}
u^f_t(x,i)+&\displaystyle\Pi_{(x,i)}\left[
\int^t_0\widehat\psi(
X_s,  Y_{s}; u^f_{t-s})ds\right]=\Pi_{(x,i)}f(X_t,  Y_t), \quad
t\ge 0.\end{array}
\end{equation}
This means that $\{\chi_t, t\ge 0\}$ can be viewed as a superprocess
 with the switched diffusion $(X_t,  Y_t)$ as its spatial motion on the space
$D\times S$ and $\widehat \psi(x,i; \cdot)$
as its (non-local) branching mechanism.
 See \cite{DKS} for a definition of superprocesses with
general non-local branching mechanisms.

\section{Main result}

It follows from \cite[Theorem 5.3]{CZ} that
the switched diffusion  $\{(X_t, Y_t), t\ge 0\}$ in $D$ has a transition density
$p(t, (x, k), (y, l))$, which is positive for all $x, y\in D$
and $k, l\in S$.  Furthermore, for any $k, l\in S$ and
$t>0$, $(x, y)\mapsto p(t, (x, k), (y, l))$ is continuous.
Let $\{P_t: t\ge 0\}$ be the transition semigroup of
$\{(X_t, Y_t), t\ge 0\}$.
For any $t>0$, $P_t$ is a compact self-adjoint operator.
Let $\{e^{\nu_kt}: k=1, 2,\cdots\}$ be all the eigenvalues of
$P_t$ arranged in decreasing order, each repeated according to its multiplicity. Then
$\lim_{k\to \infty} \nu_k= -\infty$
and the corresponding
eigenfunctions $\{\varphi_k\}$ can be chosen so that they form an orthonormal
basis of $L^2(D\times S, dx\times di)$.  All the eigenfunctions $\varphi_k$ are continuous.
The eigenspace corresponding to $e^{\nu_1t}$ is of dimension 1 and $\varphi_1$ can chosen to be strictly positive.

Let $\{P^{{\cal A}+B\cdot (N-I)}_t,t\ge 0\}$
be the Feynman-Kac semigroup defined by
$$
P^{{\cal A}+B\cdot (N-I)}_tf(x,i)
:=\Pi_{(x,i)}\left[f(X_t, Y_t)\
\exp\left(\int^t_0b(X_s,  Y_s)
(n(X_s,  Y_s)-1) ds\right)\right].
$$
Then, by \eqref{decom-M}, $P^{{\cal A}+B\cdot (N-I)}_tf(x,i)$
is the unique solution to \eqref{pde-mean} and thus
\begin{equation}\label{expX}
\bP_{\delta_{(x,i)}}\langle f, \chi_t\rangle=
P^{{\cal A}+B\cdot (N-I)}_tf(x,i).
\end{equation}

Under  the assumptions above,
$P^{{\cal A}+B\cdot (N-I)}_t$ admits a density
$\widetilde{p}(t,(x,i),(y,j))$, which is jointly continuous in $(x,y)\in D\times D$,
such that
$$
P^{{\cal A}+B\cdot (N-I)}_tf(x,i)
=\sum_{j\in S}\int_D\widetilde{p}(t, (x,i), (y,j))f(y,j)dy,
$$
for every $f\in{\cal B}^+_b(D\times S).$
$\{P^{{\cal A}+B\cdot (D-I)}_t, t\ge 0\}$ can be extended to a
strongly continuous semigroup
on $L^2(D\times S, dx\times di)$. The semigroup
$\{ P^{{\cal A}+B\cdot (N-I)}_t, t\ge 0\}$
is symmetric in $L^2(D\times S, dx\times di)$, that is
$$
\sum_{i\in S}\int_Df(x,i)P^{{\cal A}+B\cdot (N-I)}_tg(x,i)dx
=\sum_{i\in S}\int_Dg(x,i) P^{{\cal A}+B\cdot (N-I)}_tf(x,i)dx
$$
for $f,g\in
L^2(D\times S, dx\times di).$
For any $t>0$, $P^{{\cal A}+B\cdot (N-I)}_t$
is a compact self-adjoint operator.
The generator of
the semigroup $\{P^{{\cal A}+B\cdot (N-I)}_t\}$
is ${\cal A}+B\cdot (N-I)={\cal L}+B\cdot (R-I)$.

Let $\{e^{\lambda_kt}: k=1, 2,\cdots\}$ be all the eigenvalues of
$P^{{\cal A}+B\cdot (N-I)}_t$ arranged in decreasing order,
each repeated according to its multiplicity.
Then $\lim_{k\to \infty} \lambda_k = -\infty$
and the corresponding
eigenfunctions $\{\phi_k\}$ can be chosen so that they form an orthonormal
basis of $L^2(D\times S, dx\times di)$.  All the eigenfunctions $\phi_k$ are continuous.
The eigenspace corresponding to $e^{\lambda_1t}$ is of dimension 1 and $\phi_1$ can chosen to be strictly positive. For simplicity, in the remainder of this paper, we
will $\phi_1$ as $\phi$.

Throughout this paper we assume that $\{\chi_t, t\ge 0\}$ is supercritical
and $\phi$ is bounded on $D\times S$; that is, we assume the following.

\begin{assumption}\label{assume1}
$\lambda_1>0$ and its corresponding positive eigenfunction $\phi$ is bounded.
\end{assumption}

Define
\begin{equation}\label{e:dofRr}
R^\phi(x):=\left(r^\phi_{ij}(x)\right),
\quad r^\phi_{ij}(x):=r_{ij}(x)\frac{\phi(x, j))}{\phi(x, i)}
=n(x,i)\frac{p^{(i)}_j(x)\phi(x, j)}{\phi(x, i)}
\end{equation}
and
\begin{equation}\label{e:dofpi(phi)}
\pi(\phi)(x,i):=\pi(x,i;\phi)=
\sum^K_{j=1}p^{(i)}_j(x)\phi(x,j),\quad (x,i)\in D\times S.
\end{equation}

 Let $\{{\cal E}_t; t\geq 0\}$
be the minimum augmented filtration generated by
the switched diffusion   $(X, Y)$ in $D$.
 Define a
measure $\Pi_{(x,i)}^\phi$ by
\begin{equation}\label{Martingale switched diffusion}
\begin{array}{rl}
\displaystyle\frac{d\Pi_{(x,i)}^\phi}{d\Pi_{(x,i)}}\Big|_{\mathcal{E}_t}= &\displaystyle e^{-\lambda_1t}\frac{\phi(X_t, Y_{t})}{\phi(x,i)}\exp\left(\int^t_0b(X_s, Y_{s})(n(X_s, Y_s)
-1){\rm d}s\right).
\end{array}
\end{equation}
Then $\{(X, Y), \Pi_{(x,i)}^\phi\}$ is a conservative Markov process which is symmetric with respect to
the measure $\phi^2(x, i)dx\times di$.
The process $\{(X, Y), \Pi_{(x,i)}^\phi\}$ has
a transition density $p^\phi(t, (x,i),(y,j))$ with respect to $dy\times dj$ given by
$$
p^\phi(t,(x,i),(y,j))=\frac{e^{-\lambda_1t}\phi(y,j)}{\phi(x,i)}\ \widetilde{p}(t, (x, i),
(y, j)),\quad (x,i)\in D\times S.
$$
Let $\{P^{\phi}_t: t\ge 0\}$ be the transition semigroup of
$(X,  Y)$ under $\Pi_{(x,i)}^\phi$.
Then $\phi^2$ is
the unique invariant probability density of $\{P^{\phi}_t: t\ge 0\}$, that is,
for any $f\in B^+_b(D\times S)$,
$$
\sum^K_{i=1}\int_D\phi^2(x,i)P^{\phi}_tf(x,i){\rm d}x=\sum^K_{i=1}\int_Df(x,i)\phi(x,i)^2{\rm d}x.
$$
Since the infinitesimal generator of
$\{(X, Y), \Pi_{(x,i)}\}$
is ${\cal L}+\widehat B(x)\cdot(P(x)-I)$
with zero Dirichlet boundary condition on $\partial D \times S$,
it follows from  \cite[Theorem 4.2]{PR} that the generator of  $\{ (X, Y), \Pi^\phi_{(x,i)}\}$ is
\begin{eqnarray*}
&&\frac{1}{{\bf \phi}}\left[\cL({\bf u\phi})+\widehat B(x)\cdot( P(x)-I)({\bf{u\phi}})-{\bf u}(\cL({\bf \phi})+\widehat B(x)\cdot( P(x)-I){\bf{\phi}})\right]\\
&=&
\frac{1}{{\bf \phi}}\left[\cL({\bf u\phi})+\widehat B(x)\cdot( P(x)-I)({\bf{u\phi}})+B(x)\cdot(N(x)-I)({\bf{u\phi}})-\lambda_1{\bf u}\phi\right]\\
&=&\frac{1}{{\bf \phi}}\left[\cL({\bf u\phi})+
B(x)\cdot(R(x)-I)({\bf{u\phi}})\right]-\lambda_1{\bf u}\\
&=&{\cal L}^\phi {\bf u}+B(x)\cdot (R^\phi(x)-I){\bf u}-\lambda_1{\bf u},
 \end{eqnarray*}
where in
  the first equality above we used the fact that $\phi$ is an eigenfunction of
$P^{{\cal A}+B\cdot (N-I)}_t$  and \eqref{decom-M'}.

Note that
$$
\begin{array}{rll}\displaystyle B(x)\cdot (R^\phi-I)-\lambda_1
&=&\displaystyle\mbox{diag}\left(\frac{bn\pi(\phi)}{\phi}(x,1),\cdots \frac{bn\pi(\phi)}{\phi}(x,K)\right)(\widetilde P(x)-I)\\
&&\displaystyle+B(x)\left[\mbox{diag}\left(\frac{n\pi(\phi)}{\phi}(x,1),\cdots \frac{n\pi(\phi)}{\phi}(x,K)\right)-I\right]-\lambda_1
\\&=&\displaystyle\mbox{diag}\left(\frac{bn\pi(\phi)}{\phi}(x,1),\cdots \frac{bn\pi(\phi)}{\phi}(x,K)\right)(\widetilde P(x)-I).
\end{array}
$$
Thus the generator of
$\{(X, Y), \Pi^\phi_{(x,i)}\}$ is
\begin{equation}\label{transformed swithed diffu}
{\cal L}^\phi+\mbox{diag}\left(\frac{bn\pi(\phi)}{\phi}(x,1),\cdots \frac{bn\pi(\phi)}{\phi}(x,K)\right)(\widetilde P(x)-I)
\end{equation}
which is the generator of a new switched diffusion,
where
$$
{\cal L}^\phi=\left(\begin{array}{llll}
\cL^{\phi(\cdot, 1)}_1\quad &0&\cdots&\, 0\\
0\quad & \cL^{\phi(\cdot,2)}_2\,&\cdots&\, 0\\
\vdots\quad&\vdots\,&\ddots&\, \vdots\\
0\quad &0&\cdots&\, \cL^{\phi(\cdot, K)}_K\end{array}\right),
$$
$$
\cL^{\phi(\cdot, k)}_ku_k(x)=\frac{1}{\phi(x, k)}{\cL}_k\left(\phi(x, k)u_k(x)\right),
$$
$$
\widetilde P(x)=\left(\widetilde p_{ij}(x)\right)_{i,j\in S},
$$
and
$$
\widetilde p_{ij}(x)=\frac{\phi(x,i)}{n(x,i)\pi(x,i;\phi)}r^\phi_{ij}(x)
=\frac{p^{(i)}_j(x)\phi(x, j)}{\pi(x, i;\phi)},\quad i,j\in S, x\in D.
$$

For any measure $\mu$ on $D\times S$ such that $\langle\phi,\mu\rangle<\infty$, define
$$\Pi^\phi_{\phi\mu}=\frac{1}{\langle\phi,\mu\rangle}\int\phi(x,i)\Pi^\phi_{(x,i)}d\mu.$$
By \eqref{transformed swithed diffu}, the jumping intensity of $(X, Y)$ under
$\Pi^\phi_{\phi\mu}$  is $\frac{bn\pi(\phi)}{\phi}(x, i)$ at $(x,i)\in D\times S.$

\bigskip

Throughout this paper we assume that

\begin{assumption}\label{assume2} The first eigenfunction $\phi$ is bounded on $D\times S$.
The semigroup $\{P_t:t\ge 0\}$ is intrinsically ultracontractive, that is, for any $t>0$, there exists $c_t>0$ such that
$$
p(t, (x, k), (y, l))\le c_t\phi(x, k)\phi(y, l), \qquad x, y\in D, k, l\in S.
$$
\end{assumption}

It follows from \cite[Theorem 3.4]{DS} that the semigroup
$\{P^{{\cal A}+B\cdot (N-I)}_t: t\ge 0\}$ is also intrinsically ultracontractive,
that is, for any $t>0$, there exists $c_t>0$ such that
$$
\widetilde{p}(t, (x, k), (y, l))\le c_t\phi(x, k)\phi(y, l), \qquad x, y\in D, k, l\in S.
$$
As a consequence, one can easily show (see, for instance, \cite{Ban}) that
for any $t_0>0$, there exists $c>0$ such that for all $t\ge t_0$,
$$
\left|\frac{e^{-\lambda_1t}\widetilde{p}(t, (x, k), (y, l))}{\phi(x, k)\phi(y, l)}-1\right|\le
ce^{(\lambda_2-\lambda_1)t},
 \qquad x, y\in D, k, l\in S.
$$
Hence for any $\delta\in (0, 1)$,  there exists $t_0>0$ such that for
all $t\ge t_0$,
$$
\left|\frac{e^{-\lambda_1t}\widetilde{p}(t, (x, k), (y, l))}{\phi(x, k)\phi(y, l)}-1\right|\le \delta, \qquad x, y\in D, k, l\in S.
$$
Thus for any $f\in B_b(D\times S)$,
$t>t_0$ and $(x,i)\in D\times S$,
\begin{equation}\label{IU'}
\left| P^\phi_tf(x,i)-\int_{D\times S}f(y,j)\phi(y,j)^2 dydj \right|\le \delta \int_{D\times S}f(y,j)\phi(y,j)^2dydj.
\end{equation} It follows from \eqref{IU'} that
for any $f\in B^+_b(D\times S)\cap L^1(\phi^2(x,i) \, dx\times di)$,
$t>t_0$ and $(x,i)\in D\times S$,
\begin{equation}\label{IU''}
(1-\delta)\int_{D\times S}f(y,j)\phi(y,j)^2dydj\le P^\phi_tf(x,i)\le (1+\delta)\int_{D\times S}f(y,j)\phi(y,j)^2dydj.
\end{equation}

\begin{lemma}\label{martingale}
Define
\begin{equation}\label{martingale-1}
W_t(\phi):=e^{-\lambda_1 t}\langle \phi, \chi_t\rangle.
\end{equation}
 Then $\{W_t(\phi), t\ge 0\}$ is a non-negative $\bP_\mu$-martingale  for each nonzero $\mu\in M_F(D\times S)$  and
therefore there exists a limit $ W_{\infty}(\phi)\in[0,\infty)$
$\bP_{\mu}$-a.s.
\end{lemma}

\begin{proof} By the Markov property of $\chi$ and \eqref{expX},  and using the fact that $\phi$ is an eigenfunction corresponding $\lambda_1$,
 we get that for any nonzero $\mu\in M_F(D\times S)$,
 \begin{eqnarray*}
\bP_{\mu} \left[W_{t+s}(\phi) \big| \cF_t\right] &=& \frac{1}{\langle\phi,
\mu\rangle}e^{-\lambda_1 t} \bP_{\chi_t} \left[e^{-\lambda_1 s} \langle \phi, \chi_s\rangle\right]\\
& =&\frac{1}{\langle\phi, \mu\rangle}e^{-\lambda_1 t}
\left\langle  e^{-\lambda_1 s}
P^{{\cal A}+B\cdot (N-I)}_s\phi, \, \chi_t\right\rangle \\
&=&\frac{1}{\langle\phi, \mu\rangle}e^{-\lambda_1 t} \langle\phi, \, \chi_t\rangle =
W_t(\phi).
\end{eqnarray*}
This proves that $\{W_t(\phi), t\ge 0\}$ is a non-negative
$\bP_\mu$-martingale and so it has an almost sure limit
$W_\infty(\phi)\in[0,\infty)$ as $t\to \infty$.
\qed\end{proof}

We define a new kernel  $F^{\pi(\phi)}(x,i; dr)$ from $D\times S$ to $(0,\infty)$ such that for any nonnegative measurable function $f$ on $(0,\infty),$
$$\int^\infty_0f(r)F^{\pi(\phi)}(x,i;dr)=\int^\infty_0f(\pi(x,i;\phi)r)F(x,i;dr),\quad (x,i)\in D\times S.$$

Define
\begin{equation}\label{def-l}
l(x,i):=\int_0^\infty r\log^+(r)F^{\pi(\phi)}(x,i; dr).
\end{equation}

The main result of this paper is  the following.

\begin{thm}\label{maintheorem} Suppose that
$\{\chi_t; t\ge 0\}$ is a multitype superdiffusion and
that Assumptions \ref{assume1} and \ref{assume2}   hold.
Assume that  $\mu\in M_F(D\times S)$ is non-trivial.
 Then $W_\infty(\phi)$ is non-degenerate under $\bP_\mu$
if and only if
 \begin{equation}\label{LlogL-BH}
\int_{D}{\phi}(x,i)b(x,i)l(x,i)dx<\infty  \quad \hbox{for every }  i\in S,
 \end{equation}
where $l$ is defined in \eqref{def-l}.
Moreover, when \eqref{LlogL-BH} is satisfied,  $W_t(\phi)$
converges to $W_\infty(\phi)$ in $L^1$ under $\bP_\mu$.
\end{thm}

Since \eqref{LlogL-BH} does not depend on $\mu$, it is also equivalent to that
$W_\infty(\phi)$ is non-degenerate under $\bP_\mu$ for
every non-trivial
measure $\mu\in M_F(D\times S)$.

The proof of this theorem is accomplished by combining the ideas
from \cite{LPP} with the ``spine decomposition" of \cite{EK} and \cite{LRS09}.
The new feature here is that we consider a different type of branching
mechanisms. The new type of branching mechanisms considered here is
non-local as opposed to the local branching mechanisms in \cite{EK} and \cite{LRS09}.
The non-local branching mechanisms we consider
here result in a kind of {\it non-local}  immigration, as opposed to
the local immigration in \cite{LRS09}.

In the next section, we show that when $D$ is a bounded $C^{1,1}$ domain in
$\R^d$, Assumption \ref{assume2} holds. In Section \ref{s:spine}, we give
our spine decomposition of the superdiffusion $\chi$ under  a martingale change
of measure with the help of Poisson point processes. In Section \ref{s:proof},
we use this spine decomposition to prove Theorem \ref{maintheorem}.

\section{Intrinsic Ultracontractivity}\label{s:iu}

In this section, we  show that when $D$ is a bounded $C^{1, 1}$ domain in
$\R^d$, Assumption \ref{assume2} holds, that is,
the semigroup $\{P_t:t\ge 0\}$ is intrinsically
ultracontractive and the first eigenfunction is bounded.

Throughout this section, we assume  that $D$ is a bounded $C^{1, 1}$ domain in
$\R^d$. Let $p_0(t, x, y)$ be the transition density of the killed Brownian motion
in $D$. For each $i\in S$, let $p_i(t, x, y)$ be the transition density of $\xi^{i, D}_t$, the process obtained by killing the diffusion with generator ${\cL}_i$ upon exiting from $D$.

It is known (see \cite{CKP})
that there exist positive constants $C_i$, $i=1, 2, 3, 4$, such that
for all $t\in (0, 1]$, $j=0, 1, \cdots K$ and $x, y\in D$,
\begin{align}
&p_j(t, x, y)\ge C_1\left(\frac{\delta_D(x)}{\sqrt{t}}\wedge 1\right)
\left(\frac{\delta_D(y)}{\sqrt{t}}\wedge 1\right)t^{-d/2}e^{-\frac{C_2|x-y|^2}{t}},\label{e:lowerbnd}\\
&p_j(t, x, y)\le C_3\left(\frac{\delta_D(x)}{\sqrt{t}}\wedge 1\right)
\left(\frac{\delta_D(y)}{\sqrt{t}}\wedge 1\right)t^{-d/2}e^{-\frac{C_4|x-y|^2}{t}}.\label{e:upperbnd}
\end{align}
Using these we can see that there exists $C_5>0$ such that for any $t\in (0, C_4/C_2]$ and $x, y\in D$,
\begin{equation}\label{e:compofhks}
p_j(t, x, y)\le C_5p_0\left({C_2t}/{C_4}, x, y\right).
\end{equation}

It follows from \cite[Theorem 5.3]{CZ} that for any $x, y\in D$ and $k, l\in S$,
\begin{align}
&p(t, (x, k), (y, l))\nonumber\\
&\quad = \delta_{kl}p_k(t, x, y)\nonumber\\
&\qquad+\sum^\infty_{n=0}\sum_{\stackrel{1\le l_1, l_2, \dots, l_n\le K}{l_1
\neq k, l_n\neq l, l_i\neq l_{i+1}}}\int\cdots\int_{0<t_1<t_2<\cdots<t_n<t}\int_D\cdots\int_Dp_k(t_1, x, y_1)q_{kl_1}(y_1)\nonumber\\
&\qquad\quad\times p_{l_1}(t_2-t_1, y_1, y_2)q_{l_1l_2}(y_2)\cdots q_{l_nl}(y_n)\nonumber\\
&\qquad\quad\times p_l(t-t_n, y_n, y)dy_n\cdots dy_1dt_n\cdots dt_1.
\label{e:CZdensity}
\end{align}
Let $M>0$ be such that
$$
|q_{kl}(x)|\le M, \qquad x\in D, k, l\in S.
$$
Then it follows from \eqref{e:compofhks} and \eqref{e:CZdensity} that for
$t\in (0, C_4/C_2]$, $x, y\in D$ and $k, l\in S$,
\begin{align*}
&p(t, (x, k), (y, l))\\
&\le C_5p_0\left( {C_2t}/{C_4}, x, y\right)\\
&\quad +\sum^\infty_{n=0}(MKC_5)^n\int\cdots\int_{0<t_1<t_2<\cdots<t_n<t}\int_D\cdots\int_Dp_0
\left({C_2t_1}/{C_4}, x, y_1\right)\\
&\quad\quad\times
p_0\left( {C_2(t_2-t_1)}/{C_4}, y_1, y_2\right)\cdots
p_0\left({C_2(t_-t_n)}/{C_4}, y_n, y\right)dy_n\cdots dy_1dt_n\cdots dt_1\\
&\le C_5p_0\left({C_2t}/{C_4}, x, y\right)+\sum^\infty_{n=0}\frac{(MKC_5t)^n}{n\!}p_0\left({C_2t}/{C_4}, x, y\right).
\end{align*}
Thus there exists $t_0\in (0, C_4/C_2)$ such that for $t\in (0, t_0]$,
 $x, y\in D$ and $k, l\in S$,
 \begin{equation}\label{e:upbd4density}
 p(t, (x, k), (y, l))\le C_6p_0( {C_2t}/{C_4}, x, y)
 \end{equation}
for some $C_6>0$.

Now we prove a similar lower bound. It follows from \eqref{e:CZdensity}
that for any $t\in (0, 1]$, $x, y\in D$ and $k\in S$,
\begin{equation}\label{e:lowerbd1}
 p(t, (x, k), (y, k))\ge p_k(t, x, y).
\end{equation}
Now suppose $k\neq l$. Let $l_0, l_1, \cdots, l_n\in S$ with $l_i\neq l_{i+1}$, $l_0=k$,
$l_n=l$ such that $\{x\in D: q_{l_il_{i+1}}(x)>0\}$ has positive Lebesgue
measure for $i=0, 1, \cdots, n-1$.
Then it follows from \eqref{e:CZdensity} that
\begin{align*}
p(t, (x, k), (y, l))&\ge \int\cdots\int_{0<t_1<t_2<\cdots<t_n<t}\int_D\cdots\int_Dp_k(t_1, x, y_1)q_{kl_1}(y_1)\\
&\qquad \times p_{l_1}(t_2-t_1, y_1, y_2)q_{l_1l_2}(y_2)\cdots q_{l_nl}(y_n)\nonumber\\
&\qquad\times p_l(t-t_n, y_n, y)dy_n\cdots dy_1dt_n\cdots dt_1.
\end{align*}
Thus it follows from \eqref{e:lowerbnd} that there exists $C_7>0$ such that
for any $t\in (0, 1]$, $x, y\in D$,
\begin{equation}\label{e:lowerbd2}
p(t, (x, k), (y, l))\ge C_7\left(\frac{\delta_D(x)}{\sqrt{t}}\wedge 1\right).
\end{equation}
Combining \eqref{e:lowerbd1} and \eqref{e:lowerbd2} we get that for any $t\in (0, 1]$, $x, y\in D$ and $k, l\in S$,
\begin{equation}\label{e:lowerbd}
p(t, (x, k), (y, l))\ge C_8\left(\frac{\delta_D(x)}{\sqrt{t}}\wedge 1\right)
\end{equation}
for some $C_8>0$.

It follows from \eqref{e:upbd4density} and \eqref{e:lowerbd} that there exists positive constants $C_{9}<C_{10}$ such that for all $(x, k)\in D\times S$,
$$
C_{9}\delta_D(x)\le \phi(x, k)\le C_{10}\delta_D(x).
$$
Combining this with \eqref{e:upbd4density}, and using the semigroup property, we immediately get the intrinsic ultarcontractivity of
$\{P_t:t\ge 0\}$.
The boundedness of $\phi$ is an immediate consequence of the display above.

\section{Spine decomposition}
\label{s:spine}

Let ${\cal F}_t=\sigma(\chi_s;\ s\leq t)$. We define a probability
measure $\widetilde{\bP}_\mu$ by:
\begin{equation}\label{measure-change}
\frac{d\widetilde{\bP}_\mu}{d\bP_\mu}\Big|_{{\cal F}_t} =
\frac{1}{\langle\phi,\mu\rangle}W_t(\phi).
\end{equation}
The purpose of this section is to give a spine decomposition of $\{\chi_t, t\ge 0\}$
under $\widetilde{\bP}_\mu$.  This decomposition will play an important role
in  proving Theorem \ref{maintheorem}

The spine decomposition is roughly as follows:
Under $\widetilde{\bP}_\mu$,  $\{\chi_t, t\ge 0\}$
has the same law as the sum of the following two
independent measured-valued processes: the first process is a copy
of $\chi$ under $\bP_{\mu}$,  and the second process is, roughly speaking,
obtained by taking an ``immortal particle" that moves according to
the law of $\{(X,  Y),\Pi^\phi_{\phi\mu}\}$
and spins off pieces of
mass that continue to evolve according to the dynamics of $\chi$.

Define
\begin{equation}\label{e:dofeta}
\eta(x,i;\lambda)=\int_0^\infty e^{-u\lambda}uF(x,i;du),\quad \lambda\ge 0,\, (x,i)\in D\times S.
\end{equation}
We first  give a formula for the one-dimensional distribution of $\chi$
under $\widetilde{\bP}_\mu$.

\begin{thm}\label{theorem1}
Suppose $\mu\in M_F(D\times S)$ and
$g\in B^+_b(D\times S).$
Let $D_J$ be the set of jump times of $(X, Y)$. Then
\begin{align}\label{1inthm1}
&\widetilde{\bP}_\mu\left(\exp\langle-g,\chi_t \rangle\right)\nonumber\\
 &=\
\bP_\mu\big(\exp\langle - g,\chi_t\rangle\big)
\Pi_{\phi\mu}^{\phi}\left[\exp\left(\sum_{s\in D_J, 0<s\le t}\ln\left(\frac{\eta(X_{s}, Y_{s};\pi(X_{s},  Y_{s};u^g_{t-s}))}{n(X_{s}, Y_{s})}+\frac{\widetilde n(X_s,  Y_s)}{n(X_s, Y_s)}\right)\right) \right],
\end{align}
where  $u^g_{t-s}$ is the unique locally bounded positive solution
of \eqref{int} with $f$ replaced by $g$.
\end{thm}

\begin{proof} By \eqref{measure-change},
\begin{equation}\label{exp-tilde-P}\begin{array}{rl}
\displaystyle \widetilde{\bP}_\mu\big(\exp\langle-g,\chi_t \rangle\big) =&
\displaystyle\frac{e^{-\lambda_1 t}}{\langle\phi,\mu\rangle}{\bP}_\mu\left(\langle \phi,\chi_t\rangle\exp\langle-g,\chi_t \rangle\right)\\
=&\displaystyle\left.\frac{e^{-\lambda_1 t}}{\langle\phi,\mu\rangle}\frac{\partial}{\partial\theta}
{\bP}_\mu\left(\exp\langle-g-\theta\phi,\chi_t \rangle\right)\right |_{\theta=0}\\
=&\left.\displaystyle\frac{e^{-\lambda_1 t}}{\langle\phi,\mu\rangle}\frac{\partial}{\partial\theta}
\exp\langle-u^{g+\theta\phi}_t,\mu\rangle\right |_{\theta=0}\\
=&\displaystyle\frac{e^{-\lambda_1 t}}{\langle\phi,\mu\rangle}\exp\langle-u^g_t,\mu\rangle\left\langle\left.\frac{\partial}{\partial\theta}
 u^{g+\theta \phi}_t\right|_{\theta=0},\mu\right\rangle.
\end{array}
\end{equation}
Note that $\exp\langle-u^g_t,\mu\rangle =\bP_\mu\exp\langle -g,\chi_t\rangle,$
and $u^{g+\theta \phi}_t$ is the unique
locally bounded positive solution of the integral equation
$$
u^{g+\theta\phi}_t(x,i)+\Pi_{(x,i)}\left[
\int^t_0\widehat\psi(X_s,  Y_{s}; u^{g+\theta\phi}_{t-s})ds\right]=\Pi_{(x,i)}\left[(g+\theta\phi)(X_t, Y_t)\right], \quad
t\ge 0.\
$$
Taking derivative with respect to $\theta$ on both sides of the above equation, and then letting $\theta=0$, we have that $v_t(x,i):=\left.\frac{\partial}{\partial\theta}
 u^{g+\theta \phi}_t\right|_{\theta=0}$ satisfies
 \begin{align}\label{int-v}
 &v_t(x,i)-
 \Pi_{(x,i)}\int^t_0b(X_s, Y_{s})\left(n(X_s,  Y_{s})-1\right)v_{t-s}(X_s, Y_{s})ds\nonumber\\
 &+\ \Pi_{(x,i)}\int^t_0b(X_s,  Y_{s})\left(m(X_s,  Y_{s})-\eta\left(X_s, Y_{s}; \pi(X_s,  Y_{s}; u^{g}_{t-s})\right)\right)
\sum^K_{j=1}p^{(Y_{s})}_{j}(X_s)v_{t-s}(X_s,j)ds\nonumber\\
 &=\ \Pi_{(x,i)}\left[\phi(X_t,  Y_t)\right].
 \end{align}
Let
\begin{equation}\label{def-J}
J((x, k), d(y, l))=\delta(x-y)q_{kl}(x)1_{\{k\neq l\}}dydl,
\qquad (x, k)\in D\times S,
\end{equation}
where $dl$ stands for the counting measure on $S$.
Then $(J((x, k), d(y, l)), t)$ is a L\'evy system of $(X, Y)$.
Define
\begin{equation}\label{def-F}
F(t-s, (x, i), (y, j)):=\ln \left(\frac{\eta(x, i; \pi(x, i; u^g_{t-s})}{n(x, i)}-
\frac{m(x, i)}{n(x, i)}+1\right)1_{i\neq j}.
\end{equation}
Clearly, $F\le 0$.
We would like to apply Lemma
\ref{G-nonlocal-FK}  with $\xi=(X, Y)$, $q(t-s, (x,i))=b(x,i)(n(x,i)-1)$, $J$  given by \eqref{def-J} and $F$ given by \eqref{def-F}.  Since $q_{ij}(x)$,
$i,j\in S$,  are bounded in $D$ and $D$ has finite Lebesgue measure, we have $\sup_{(x,i)\in D\times S}J((x,k), D\times S)<\infty$. By Remark \ref{rem6.2}(iii), conditions \eqref{e:6.1} and \eqref{e:6.2} are satisfied.  Thus we can apply Lemma \ref{G-nonlocal-FK} to get
\begin{align}\label{non-localFK'}
&v_t(x,i)\nonumber\\
&=\ \Pi_{(x,i)}\left[\exp\left\{\sum_{s\in D_J, 0<s\le t}\ln\left(\frac{\eta(X_{s},  Y_{s};\pi(X_{s}, Y_{s};u^g_{t-s}))}{n(X_{s}, Y_{s})}-\frac{m(X_s, Y_s)}{n(X_s, Y_s)}+1\right)\right.\right.\nonumber\\
&\qquad\qquad\qquad\left.\left.+\int^t_0b(X_s, Y_{s})
(n(X_s, Y_{s})-1){\rm d}s\right\}\phi(X_t,  Y_t)\right]\nonumber\\
&=\ e^{\lambda_1t}\phi(x,i)\Pi_{(x,i)}^{\phi}\left[\exp\left\{\sum_{s\in D_J, 0<s\le t}\ln\left(\frac{\eta(X_{s}, Y_{s};\pi(X_{s}, Y_{s};u^g_{t-s}))}{n(X_{s}, Y_{s})}+\frac{\widetilde n(X_s,  Y_s)}{n(X_s, Y_s)}\right)\right\}\right].
\end{align}
Combining \eqref{exp-tilde-P} and \eqref{non-localFK'},
we obtain
\begin{align*}
&\widetilde{\bP}_\mu\big(\exp\langle-g,\chi_t \rangle\big)\\
&=\bP_\mu\big(\exp\langle - g,\chi_t\rangle\big)\\
&\quad\cdot\Pi_{\phi\mu}^{\phi}\left[\exp\left\{\sum_{s\in D_J, 0<s\le t}\ln\left(\frac{\eta(X_{s}, Y_{s};\pi(X_{s},  Y_{s};u^g_{t-s}))}{n(X_{s},  Y_{s})}+\frac{\widetilde n(X_s, Y_s)}{n(X_s, Y_s)}\right)\right\}\right].
\end{align*}
\qed
\end{proof}

Define
\begin{equation}\label{def-n}\widetilde F(x,i;du)=\frac{1}{n(x,i)}\left(\widetilde n(x,i)\delta_0+I_{(0,\infty)}u\, F(x,i;du)\right).
\end{equation}
Then, by \eqref{def-tilde-d} and \eqref{e:dofeta},
$\widetilde F(x,i;\cdot)$ is a probability measure
on $[0,\infty)$ for any $(x,i)\in D\times S$ and
$$
\frac{\eta(x,i;\lambda)}{n(x, i)}+\frac{\widetilde n(x,i)}{n(x,i)}=
\int_{[0,\infty)}e^{-u\lambda}\widetilde F(x,i;du) \quad \hbox{for every } \lambda\ge 0.
$$
Thus we may rewrite
 \eqref{1inthm1} as
\begin{equation}\label{1inthm1'}\begin{array}{rl}&\displaystyle\widetilde{\bP}_\mu\big(\exp\langle-g,\chi_t \rangle\big)\\
=&\displaystyle \bP_\mu\big(\exp\langle - g,\chi_t\rangle\big)\cdot\Pi_{\phi\mu}^{\phi}\left[\prod_{s\in D_J, 0<s\le t}\int^\infty_0\exp(-u\pi(X_{s}, Y_{s};u^g_{t-s}))\widetilde F(X_{s}, Y_{s};du)\right].
\end{array}\end{equation}

 From \eqref{1inthm1'} we  see that the superdiffusion
$\{\chi_t, t\ge 0; \widetilde{\bP}_\mu\}$ can be decomposed
into two independent parts.  The first part is
 a copy of the original superdiffusion and the second part is an
immigration process. To describe the second part precisely, we
need to introduce another measure-valued process $\{\widehat{\chi}_t, t\ge 0\}$.
Now we construct the measure-valued process $\{\widehat{\chi}_t, t\ge 0\}$ as
follows:

\begin{description}
\item{(i)}
Suppose that $(\widehat X, \widehat Y)=\{(\widehat X_t,  \widehat Y_t), t\ge 0\}$ is defined
on some probability space $(\Omega, \bP_{\mu, \phi})$, and $(\widehat X, \widehat Y)$ has the same law as $((X, Y);
\Pi^\phi_{\phi\mu})$.  $(\widehat X, \widehat Y)$ serves as the spine or the
immortal particle, which visits every part of $D\times S$ for large times since it is an ergodic process.
Let $D_J$ be the set of jump points of $(\widehat X, \widehat Y)$. $D_J$ is countable.

\item{(ii)}
Conditioned on $s\in D_J$, a measure-valued process $\chi^s$ started at $m_s\delta_{({\widehat X}_s,l)} (l\in S)$ is immigrated at the space position ${\widehat X}_s$ and the new immigrated particles choose their types independently
according to the (nonrandom) distribution $\pi(x,i; \cdot)$. We suppose $\{m_s;s\in D_J\}$ is also defined on $(\Omega, \bP_{\mu, \phi})$ such that, given $s\in D_J$ and $(\widehat X_s, \widehat Y_s)$, the distribution of $m_s$ is   $\widetilde F({\widehat X}_s, {\widehat Y}_s;dr)$.

\item{(iii)}
Once the particles are in the system, they begin to move and
branch according to the $((X, Y),\, \widehat \psi(x,i,\cdot))$-superprocess
independently.
\end{description}

We use $(\chi^{s}_t,\ t\ge s )$ to denote the measure-valued
process generated by the  mass immigrated at time $s$
and spatial position ${\widehat X}_s$.
 Conditional on
$\{({\widehat X}_t, {\widehat Y}_t),  t\ge 0; m_s, s\in D_J\}$, $\{\chi^{s}_t, t\ge s\}$ for different $s\in D_J$ are
independent $((X, Y),\, \widehat \psi(x,i,\cdot))$-superprocesses. Set
\begin{equation}\label{X-hat}
\widehat{\chi}_t=\sum_{s\in(0,t]\cap D_J}\chi^s_t.
\end{equation}
The Laplace functional of $\widehat{\chi}_t$ is described in
the following proposition.

\begin{proposition}\label{prop}
The Laplace functional of $\widehat{\chi}_t$ under $\bP_{\mu,\phi}$
is equal to
$$
\Pi_{\phi\mu}^{\phi}\left\{\prod_{s\in(0,t]\cap D_J}\int_{[0,\infty)}\exp\left(-r\pi(X_s, Y_s; u^g_{t-s})\right)\widetilde F({X}_s, {Y}_s;dr)
\right\}.
$$
\end{proposition}

\begin{proof}
For any $g\in B_b^+(D\times S)$, using \eqref{Laplace}, we have
\begin{eqnarray*}
\bP_{\mu, \phi}\left[\exp(-\langle g,\widehat{\chi}_t\rangle)\right]
&=&\bP_{\mu, \phi}\left\{ \bP_{\mu, \phi}
\left[\exp(-\sum_{\sigma\in(0,t]\cap D_J}\langle g,
\chi_t^\sigma\rangle)\bigg|\sigma(({\widehat X}, {\widehat Y}), m)\right]\right\}\\
&=&\bP_{\mu, \phi}\left[\prod_{s\in(0,t]\cap D_J}
\exp\left(-m_s\pi(\widehat X_s, \widehat Y_s,u^g_{t-s})\right)\right]\\
&=&\bP_{\mu, \phi}\left\{\bP_{\mu,
\phi}\left[\prod_{s\in(0,t]\cap D_J}\exp\left(-m_s\pi(\widehat X_s, \widehat Y_s, u^g_{t-s})\right)
\bigg|\sigma(({\widehat X},{\widehat Y}))\right]\right\}\\
&=&\Pi_{\phi\mu}^{\phi}\left\{\prod_{s\in(0,t]\cap D_J}\int_{[0,\infty)}\exp\left(-r\pi(X_s, Y_s, u^g_{t-s})\right)\widetilde F({X}_s, {Y}_s;dr)
\right\}.
\end{eqnarray*}
\qed\end{proof}

Without loss of generality,
we suppose $\{\chi_t, t\ge 0; \bP_{\mu,\phi}\}$ is a multitype superdiffusion
defined on $(\Omega, \bP_{\mu, \phi})$,
having the same law as $\{\chi_t, t\ge 0;\bP_{\mu}\}$ and
independent of $\widehat\chi=\{\widehat \chi_t, t\ge 0\}$. Proposition \ref{prop} says that we have the following
decomposition of $\{\chi_t,t\ge 0\}$ under $\widetilde \bP_{\mu}$: for any
$t>0$,
\begin{equation}\label{decomp}
({\chi}_t, \widetilde \bP_{\mu}) = (\chi_t+\widehat{\chi}_t,\
\bP_{\mu,\phi})\quad\mbox{ in distribution}.
\end{equation}
Since $\{\chi_t, t\ge 0;\widetilde \bP_{\mu}\}$ is generated from the
time-homogeneous  Markov process $\{\chi_t, t\ge 0; \bP_{\mu}\}$ via a
non-negative martingale multiplicative functional, $\{\chi_t,t\ge 0;
\widetilde \bP_{\mu}\}$ is also a time-homogeneous Markov process (see
\cite[Section 62]{Sharpe}). From the construction of $\{
\widehat\chi_t, t\ge 0;\bP_{\mu,\phi}\}$ we see that $\{
\widehat{\chi}_t, t\ge 0;\bP_{\mu,\phi}\}$ is a time-homogeneous
Markov process. For a rigorous proof of  $\{ \widehat{\chi}_t, t\ge
0;\bP_{\mu,\phi}\}$ being a time-homogeneous Markov process, we refer
our readers to \cite{E}. Although the paper \cite{E} dealt with the
representation of the superprocess conditioned to stay alive
forever, one can check that the arguments there work in our case.
Therefore, \eqref{decomp} implies the following.

\begin{thm}\label{t:spine}
\begin{equation}\label{decomp2}
\{{\chi}_t, t\ge 0; \widetilde \bP_{\mu}\} = \{\chi_t+\widehat{\chi}_t, t\ge 0; \bP_{\mu,\phi}\}\quad\mbox{ in law}.
\end{equation}
\end{thm}

 \section{$L\log L$ criterion}\label{s:proof}

In this section, we give a proof of the main result of this paper, Theorem \ref{maintheorem}.
First, we make some preparations.

\begin{proposition}\label{equi-deg}
Let  $h(x,i)=\frac{1}{\phi(x,i)}\bP_{\delta_{(x,i)}}(W_{\infty}(\phi))$. Then
\begin{description}
\item{\rm (i)}  $h$ is a non-negative invariant function for the
process $((X, Y); \Pi^{\phi}_{(x,i)})$.

\item{\rm (ii)} Either $W_{\infty}$ is non-degenerate under $\bP_{\mu}$ for
all nonzero $\mu\in M_F(D\times S)$
or $W_{\infty}$ is degenerate under $\bP_{\mu}$
for all $\mu\in M_F(D\times S)$.
\end{description}
\end{proposition}

\begin{proof} (i)  By the Markov property of $\chi$,
\begin{eqnarray*}
h(x,i)&=&\frac{1}{\phi(x,i)}\bP_{\delta_{(x,i)}}\left[\lim_{s\to\infty} \langle
e^{-\lambda_1(t+s)}\phi,
\chi_{t+s}\rangle\right]\\
&=&\frac{e^{-\lambda_1 t}}{\phi(x,i)}\bP_{\delta_{(x, i)}}\left[\bP_{\chi_t}
(\lim_{s\to\infty}\langle e^{-\lambda_1s}\phi,\chi_s\rangle)\right]\\
&=&\frac{e^{-\lambda_1 t}}{\phi(x,i)}\bP_{\delta_{(x,i)}}\left[\bP_{\chi_t}(
W_{\infty})\right]\ =\ \frac{e^{-\lambda_1
t}}{\phi(x,i)}\bP_{\delta_{(x,i)}}\left[\langle
(h\phi), \chi_t\rangle\right]\\
&=&\frac{e^{-\lambda_1 t}}{\phi(x,i)}P^{{\cal A}+B\cdot(N-I)}_t(h\phi),\quad x\in D.
\end{eqnarray*}
By the definition of $\Pi^{\phi}_{(x,i)}$, we get that
$h(x,i)=\Pi^{\phi}_{(x,i)}[h(X_t, Y_t)]$. So  $h$ is an invariant function of
the process $((X, Y); \Pi^{\phi}_{(x, i)})$. The non-negativity of $h$ is
obvious.

(ii) Since $h$ is non-negative and invariant, if there exists $(x_0, i)\in
D\times S$ such that $h(x_0,i)=0$, then $h\equiv 0$ on $D\times S$. Since
$\bP_{\mu}(W_{\infty}(\phi))=\langle h\phi, \mu\rangle,$ we then have
$\bP_{\mu}(W_{\infty}(\phi))=0$ for any $\mu\in M_F(D\times S)$.  If $h>0$ on
$D\times S$, then $\bP_{\mu}(W_{\infty}(\phi))>0$ for any
nonzero $\mu\in M_F(D\times S)$.
\qed\end{proof}

Using Proposition \ref{equi-deg} we see that, to prove Theorem
\ref{maintheorem}, we only need to consider the case
$d\mu={\phi}(x,i)dxdi$, where $di$ is the counting measure on $S$. So in the remaining part of this
paper we will always suppose that $d\mu={\phi}(x,i)dxdi$.

Recall from \eqref{e:dofpi(phi)} and \eqref{def-l} that
$$
\pi(x,i;\phi)=\sum^K_{j=1}p^{i}_j(x)\phi(x,j),\quad (x,i)\in D\times S
$$
and
$$
l(x,i)=\int_0^\infty r\log^+(r)F^{\pi(\phi)}(x,i; dr)=\int_0^\infty r\pi(x, i, \phi)
\log^+(r\pi(x, i, \phi))F(x,i; dr).
$$

\begin{lemma}\label{lemma1}
Let $(m_t;\ t\in D_J)$ be the Poisson point process constructed in
Section 4, given the path of $(\widehat X_s, \widehat Y_s), s\ge 0$.   Define
$$
\sigma_0=0,\quad \sigma_i=\inf\{s\in D_J;\ s>\sigma_{i-1},\
m_s\pi(\widehat X_s, \widehat Y_s;\phi)>1\}, \quad\eta_i=m_{\sigma_i},\quad
i=1,2,\cdots
$$

\begin{description}
\item{\rm (i)} If $\displaystyle\sum^K_{i=1}\int_D{\phi}(y,i)b(y,i) l(y,i)dy<\infty$, then
\begin{equation}\label{e:5.1'}
\sum_{s\in D_J}\mbox{e}^{-\lambda_1s}m_s\pi(\widehat X_s, \widehat Y_s;\phi) < \infty, \quad
\bP_{\mu,\phi}\mbox{-a.s.}
\end{equation}

\item{\rm (ii)} If $\displaystyle \sum^K_{i=1}\int_D{\phi}(y,i)b(y,i)l(y,i)dy=\infty$,  then
\begin{equation}\label{e:5.1}
\limsup_{i\rightarrow\infty}e^{-\lambda_1\sigma_i}\eta_i
\pi(\widehat X_{\sigma_i}, \widehat Y_{\sigma_i};\phi) =\infty,\quad \bP_{\mu,\phi}\mbox{-a.s.}
\end{equation}
\end{description}
\end{lemma}

\begin{proof}
Since $\phi$ is bounded from above, $\sigma_i$ is strictly
increasing with respect to $i$.

(i) Suppose that $\sum_{i=1}^K\int_D\phi(y,i)b(y,i) l(y,i)dy<\infty$.
For any $\varepsilon>0$, we write the sum above as
\begin{align}\label{sum}
&\sum_{s\in \mathcal D_J}\mbox{e}^{-\lambda_1s}m_s\pi(\widehat X_s, \widehat Y_s;\phi)\nonumber\\
&=\sum_{s\in \mathcal D_J}
\mbox{e}^{-\lambda_1s}m_s\pi(\widehat X_s, \widehat Y_s;\phi)
1_{\{m_s\pi(\widehat X_s, \widehat Y_s;\phi)\le\mbox{e}^{\varepsilon
s}\}}\nonumber\\
&\quad +\sum_{s\in \mathcal
D_J}\mbox{e}^{-\lambda_1s}m_s\pi(\widehat X_s, \widehat Y_s;\phi)
1_{\{m_s\pi(\widehat X_s, \widehat Y_s;\phi)>\mbox{e}^{\varepsilon
s}\}}\nonumber\\
&=\sum_{s\in \mathcal
D_J}\mbox{e}^{-\lambda_1s}m_s\pi(\widehat X_s, \widehat Y_s;\phi) 1_{\{
\pi(\widehat X_s, \widehat Y_s;\phi)m_s\le \mbox{e}^{\varepsilon s}
\}}\nonumber\\
&\quad+\sum_{i=1}^\infty\mbox{e}^{-\lambda_1\sigma_i}\eta_i
\pi(\widehat X_{\sigma_i}, \widehat Y_{\sigma_i};\phi)
1_{\{\eta_i\phi(\pi(\widehat X_{\sigma_i}, \widehat Y_{\sigma_i};\phi) >\mbox{e}^{\varepsilon
\sigma_i}\}}\nonumber\\&=I+II.
\end{align}
Note that the jumping intensity of
$\{(\widehat X, \widehat Y), \bP_{\mu,\phi}\}$
is $\frac{bn\pi(\phi)}{\phi}(x, i)$
at $(x,i)\in D\times S$.
Thus
\begin{eqnarray*}
&&\sum_{i=1}^\infty \bP_{\mu,\phi}
\left(\eta_i\pi(\widehat X_{\sigma_i}, \widehat Y_{\sigma_i};\phi)  > \mbox{e}^{\varepsilon
\sigma_i}\right)\\
&=&\sum_{i=1}^\infty
\bP_{\mu,\phi}\left[\bP_{\mu,\phi}\left(\eta_i
\pi(\widehat X_{\sigma_i}, \widehat Y_{\sigma_i};\phi) >\mbox{e}^{\varepsilon
\sigma_i}\big|\sigma({\widehat X},{\widehat Y})\right)\right]\\
&=&\bP_{\mu,\phi}\left[\bP_{\mu,\phi}\left(\sum_{i=1}^\infty
1_{\{\eta_i>\mbox{e}^{\varepsilon \sigma_i}
\pi(\widehat X_{\sigma_i}, \widehat Y_{\sigma_i};\phi) ^{-1}\}}\Big|
\sigma({\widehat X}, {\widehat Y})\right)\right]\\
&=&\Pi^{\phi}_{\phi\mu}\left[\int_0^\infty(bn\pi(\phi)/\phi)(X_{s}, Y_{s})
\left(\int^{\infty}_{\pi(X_{s},  Y_{s};\phi) ^{-1}\mbox{e}^{\epsilon s}}
\widetilde F(X_s, Y_s; dr)\right)ds\right].
\end{eqnarray*}
Recall that under $\Pi^{\phi}_{\phi\mu}$, $(X, Y)$ starts at the
invariant measure $\phi^2(x,i)dxdi$. By the definition of $\widetilde F$ given in \eqref{def-n},
\begin{eqnarray*}
&&\sum_{i=1}^\infty \bP_{\mu,\phi} \left(\eta_i\pi(X_{\sigma_i}, Y_{\sigma_i};\phi)  > \mbox{e}^{\varepsilon \sigma_i}\right)\\
&=&\int_0^\infty
ds\sum^K_{j=1}\int_Ddy(b\phi)(y,j)\int^{\infty}_{\pi(y,j;\phi)^{-1}
\mbox{e}^{\epsilon s}}\pi(y,j;\phi)r\, F(y,j; dr)\\
&=& \sum^K_{j=1}\int_D(b \phi)(y,j)dy\int_{\pi(y,j;\phi)^{-1}}^\infty
\pi(y,j;\phi)r\, F(y,j; dr)\int^{\frac{\ln (r\pi(y,j;\phi))}{\epsilon}}_{0}ds\\
&=&\varepsilon^{-1}\sum^K_{j=1}\int_D(b{\phi})(y,j)l(y,j)dy.
\end{eqnarray*}
 By the assumption that $\sum^K_{j=1}\int_D(b{\phi})(y,j) l(y,j)dy<\infty$ and
the Borel-Cantelli Lemma,  we get
\begin{equation}\label{io}
\bP_{\mu,\phi}\Big(\eta_i\pi(\widehat X_{\sigma_i}, \widehat Y_{\sigma_i};\phi)>\mbox{e}^{\varepsilon \sigma_i} \mbox{ i. o.}\Big)=0
\end{equation}
for all $\varepsilon>0$,  which implies that
\begin{equation}\label{big}
II<\infty.\quad \bP_{\mu,\phi}\mbox{-a.s.}
\end{equation}
Meanwhile for $\varepsilon<\lambda_1$,
$$
\begin{array}{rl}
\bP_{\mu,\phi}I&=\displaystyle \bP_{\mu,\phi} \left[\sum_{s\in\mathcal
D_J} \mbox{e}^{-\lambda_1s}m_s\pi(\widehat X_s,  \widehat Y_s;\phi)  1_{\{m_s\le
\mbox{e}^{\varepsilon s} \pi(\widehat X_s, \widehat Y_s;\phi) ^{-1}\}}\right]\\
&=\displaystyle \Pi_{\phi\mu}^\phi\int_0^\infty dt
\mbox{e}^{-\lambda_1t}\int_0^{\pi(X_t, {Y}_t;\phi)^{-1}
\mbox{e}^{\varepsilon t}}\frac{bn\pi(\phi)}{\phi}(X_t,  Y_t)\pi(X_t, {Y}_t;\phi)r\widetilde F(X_t, {Y}_t;dr)\\
&\le \displaystyle
\Pi_{\phi\mu}^\phi\int_0^\infty dt
\mbox{e}^{-(\lambda_1-\varepsilon)t}\int_0^{\infty}\frac{b\pi(\phi)}{\phi}(X_t, Y_t)rF(X_t, {Y}_t; dr),
\end{array}
$$
 where for the  inequality above we used the fact
that $r\le\pi(X_t, {Y}_t,\phi)^{-1} \mbox{e}^{\varepsilon t}$ implies that
$r\pi(X_t, {Y}_t,\phi)\le \mbox{e}^{\varepsilon t}$. By the assumption that
$\sup_{(x,i)\in D\times S}\int_0^\infty r F(x,i, dr)<\infty$, we have
$$
\begin{array}{rll}\bP_{\mu,\phi}I&\le&\displaystyle\frac{1}{\lambda_1-\epsilon}\sum^K_{i=1}\int_D b(y, i)\pi(y,i;\phi)\phi(y,i)\int_0^\infty rF(y,i, dr)dy\\
&\le&\displaystyle
\frac{1}{\lambda_1-\epsilon}\|\int_0^\infty rF(y,i, dr)\|_\infty
\|b\|_\infty<\infty.\end{array}
$$
Thus
\begin{equation}\label{small}
I<\infty,\quad \bP_{\mu,\phi}\mbox{-a.s.}
\end{equation}
Combining \eqref{sum}, \eqref{big} and \eqref{small},
we obtain \eqref{e:5.1'}.

(ii) Next, we assume $ \sum^K_{i=1}\int_D(b{\phi})(y,i)l(y,i)dy=\infty$.
To establish \eqref{e:5.1}, it suffices to show that for  any $L>0$,
\begin{equation}\label{>K}
\limsup_{i\rightarrow\infty}e^{-\lambda_1\sigma_i}\eta_i
\pi(\widehat X_{\sigma_i}, \widehat Y_{\sigma_i};\phi)
>L,\quad \bP_{\mu,\phi}\mbox{-a.s.}
\end{equation}
Put $L_0:=1\vee(\max_{(x,i)\in D\times S}\phi(x,i))$. Then for $L\ge L_0$,
$$
 L\inf_{(x,i)\in D\times S}\phi(x,i)^{-1}\geq 1.
$$
Note that for any $T\in (0,\infty)$, conditional on
$\sigma({\widehat X}, {\widehat Y})$,
$$
\sharp\{i: \sigma_i\in(0, T];
 \eta_i>L\pi(\widehat X_{\sigma_i},
\widehat Y_{\sigma_i};\phi)^{-1}\mbox{e}^{\lambda_1\sigma_i}\}
$$
is a Poisson random variable with parameter
$$
\int_0^Tdt(b\pi(\phi)/\phi)(\widehat X_t,{\widehat Y}_t)
 \int_{L\pi(\widehat X_t, {\widehat Y}_t;\phi)^{-1}
\mbox{e}^{\lambda_1t}}^\infty r
F(\widehat X_t, {\widehat Y}_t; dr).
$$
Since $(\widehat X,  \widehat Y; \bP_{\mu,\phi})$
has the same law as $(X,  Y; \Pi^{\phi}_{\mu\phi})$, we have
$$
\begin{array}{rl}
&\displaystyle \bP_{\mu,\phi}\int_0^Tdt
\frac{b\pi(\phi)}{\phi}(\widehat X_t, {\widehat Y}_t)
 \int_{L\pi(\widehat X_t, {\widehat Y}_t;\phi)^{-1}\mbox{e}^{\lambda_1t}}^\infty
r F(\widehat X_t, {\widehat Y}_t;
dr)\\
=&\displaystyle\int^T_0dt\sum^K_{j=1}\int_Ddy(b\pi(\phi)\phi)(y,j)
 \int_{L\pi(y,j;\phi)^{-1}\mbox{e}^{\lambda_1t}}^\infty
r F(y,j; dr)<\infty,\end{array}
$$
thus
$$
\int_0^Tdt\frac{b\pi(\phi)}{\phi}(\widehat X_t, {\widehat Y}_t)
 \int_{L\pi(\widehat X_t, {\widehat Y}_t;\phi)^{-1}\mbox{e}^{\lambda_1t}}^\infty
rF(\widehat X_t, {\widehat Y}_t;
dr)<\infty,\quad \bP_{\mu,\phi}\mbox{-a.s.}
$$
Consequently we have
\begin{equation}\label{number-finite}
\sharp\Big\{i: \sigma_i\in(0,T];
\eta_i>L\pi(\widehat X_t, {\widehat Y}_t;\phi)^{-1}\mbox{e}^{\lambda_1\sigma_i}
\Big\}<\infty,\quad
\bP_{\mu,\phi}\mbox{-a.s.}
\end{equation}
So, to prove \eqref{>K},
 we need to prove
$$
\int_0^\infty dt\frac{b\pi(\phi)}{\phi}(\widehat X_t, {\widehat Y}_t)
\int_{L\pi(\widehat X_t, {\widehat Y}_t;\phi)^{-1}\mbox{e}^{\lambda_1t}}^\infty
rF(\widehat X_t, {\widehat Y}_t;
dr)=\infty,\quad \bP_{\mu,\phi}\mbox{-a.s.}
$$
which is equivalent to
\begin{eqnarray}\label{=infinity}
\int_0^\infty dt\frac{b\pi(\phi)}{\phi}(X_t, {Y}_t)
\int_{L\pi(X_t, Y_t;\phi)^{-1}\mbox{e}^{\lambda_1t}}^\infty
r F(X_t, {Y}_t; dr)=\infty,\quad \Pi^{\phi}_{\phi\mu}\mbox{-a.s.}
\end{eqnarray}

For this purpose we first prove that
\begin{equation}\label{mean=infty}
\Pi^{\phi}_{\phi\mu}\left[\int_0^\infty dt\frac{b\pi(\phi)}{\phi}(X_t, {Y}_t)
\int_{L\pi(X_t,  Y_t;\phi)^{-1}\mbox{e}^{\lambda_1t}}^\infty
rF(X_t, {Y}_t; dr)\right]=\infty.
\end{equation}
Applying Fubini's theorem, we get
\begin{eqnarray*}
&&\Pi^{\phi}_{\phi\mu}\left[\int_0^\infty dt\frac{b\pi(\phi)}{\phi}(X_t, {Y}_t)
\int_{L\pi(X_t,  Y_t;\phi)^{-1}\mbox{e}^{\lambda_1t}}^\infty
r F(X_t, {Y}_t; dr)\right]\nonumber\\
&=& \sum^K_{j=1}\int_Db(y,j)\pi(y,j;\phi){\phi}(y,j)dy\int_0^\infty dt
\int_{L\pi(y,j;\phi)^{-1}\mbox{e}^{\lambda_1t}}^\infty r F(y,j; dr)\\
&=&\sum_{j=1}^K \int_Db(y,j)\pi(y,j;\phi){\phi}(y,j)dy
\int_{L \pi(y,j;\phi)^{-1}}^\infty
rF(y,j;dr)\int_0^{\frac{1}{\lambda_1}
\ln(\frac{r\pi(y,j;\phi)}{L})}dt\\
&=&\sum_{j=1}^K\frac{1}{\lambda_1}\int_Db(y,j)\pi(y,j;\phi)
{\phi}(y,j)dy
\int_{L\pi(y,j;\phi)^{-1}}^\infty
(\ln[r\pi(y,j;\phi)]-\ln L)r F(y,j; dr)\\
 &\ge&\sum_{j=1}^K\frac{1}{\lambda_1}\int_Db(y,j)\pi(y,j;\phi)
{\phi}(y,j)dy\left[ \int_{L\pi(y,j;\phi)^{-1}}^\infty
r\ln[r\pi(y,j;\phi)]F(y, j; dr)-A\right]\\
&=&\sum_{j=1}^K\frac{1}{\lambda_1}\int_D(b{\phi})(y,j)dy
\int_{L}^\infty
r\ln r\,F^{\pi(\phi)}(y,j; dr) -\sum_{j=1}^K
\frac{A}{\lambda_1}\int_D(b{\phi})(y,j)\pi(y, j;\phi)dy,
\end{eqnarray*}
for some constant $A>0$, where in the inequality we used the
facts that $L\pi(y,j;\phi)^{-1}>1$ for any $(y,j)\in D\times S$
and $\sup_{(y, j)\in
D\times S}\int^\infty_1rF(y,j; dr)<\infty$. It is easy to see that
$$
\sum_{j=1}^K
\frac{A}{\lambda_1}\int_D(b{\phi})(y,j)\pi(y, j;\phi)dy\le\frac{A}{\lambda_1}\|b\|_\infty<\infty.
$$
Since
$$
\sum_{j=1}^K\int_D(b{\phi})(y,j)dy\int_1^\infty r\ln r\,F^{\pi(\phi)}(y,j; dr)=\infty,
$$
and
\begin{eqnarray*}&&
\sum_{j=1}^K\int_D(b{\phi})(y,j)dy
\int_1^Lr\ln r\,F^{\pi(\phi)}(y,j; dr)\\
&\leq&L\ln L\sum^K_{j=1}\int_D (b{\phi})(y,j)F(y,j;
[\|\phi\|_{\infty}^{-1},\infty))dy<\infty,
\end{eqnarray*}
we get that
$$
\sum_{j=1}^K\int_D(b{\phi})(y,j)dy
\int_L^\infty  r\ln r\,F^{\pi(\phi)}(y,j; dr)=\infty,
$$
and therefore, \eqref{mean=infty} holds.

 By \eqref{IU''}, there exists constant $t_0>0$ such that for
any $t>t_0$ and any
$f\in B^{+}_b(D\times S)$,
\begin{eqnarray}\label{domi-p}
\frac{1}{2}\int_{D\times S}\phi^2(y,j)f(y,j)dydi&\leq&
\int_{D\times S}p^{\phi}(t, (x,i), (y,j))f(y,j)dydi\nonumber\\
&\leq&
2\int_{D\times S}\phi^2(y,j)f(y,j)dydi
\end{eqnarray}
holds for any $ (x, i)\in D\times S$.  For $T>t_0$, we define
$$
\xi_T=\int_0^T dt\frac{b\pi(\phi)}{\phi}(X_t, {Y}_t)
 \int_{L\pi(X_t,  Y_t; \phi)^{-1}\mbox{e}^{\lambda_1t}}^\infty
r F(X_t,  {Y}_t;dr)
$$
and
$$
A_T=\sum_{j=1}^K
 \int_{t_0}^Tdt\int_D(b{\phi})(y,j)dy\int_{L\mbox{e}^{\lambda_1t}}^\infty
r F^{\pi(\phi)}(y,j; dr).
$$
Our goal is to prove \eqref{=infinity}, which is equivalent to
\begin{eqnarray}
\xi_\infty:=\int_0^\infty dt\frac{b\pi(\phi)}{\phi}(X_t, {Y}_t)
 \int_{L\pi(X_t, Y_t;\phi)^{-1}\mbox{e}^{\lambda_1t}}^\infty
r F(X_t, {Y}_t; dr)=\infty, \qquad \Pi_{\phi\mu}^\phi\mbox{-a.s.}
\end{eqnarray}
Since $\left\{\xi_\infty=\infty\right\}$ is an invariant event, by
 the ergodic property of $\{(X,  Y),\Pi^{\phi}_{\phi\mu}\}$, it is
enough to prove
\begin{equation}\label{positive-prob}
\Pi^\phi_{\phi\mu}\left(\xi_\infty=\infty\right)>0.
\end{equation}
Note that
\begin{equation}\label{domi-mean}
\Pi_{\phi\mu}^\phi \xi_T=\sum^K_{j=1}\int_0^Tdt\int_D
(b{\phi})(y,j)dy
 \int_{L\mbox{e}^{\lambda_1t}}^\infty
r F^{\pi(\phi)}(y,j; dr)\ge A_T
\end{equation}
and
\begin{align}\label{mean-limit}
&\lim_{T\rightarrow\infty}\Pi_{\phi\mu}^\phi \xi_T\ge
A_\infty=\sum^K_{j=1}
 \int_{t_0}^\infty dt\int_D(b{\phi})(y,j)dy
\int_{L\mbox{e}^{\lambda_1t}}^\infty
rF^{\pi(\phi)}(y,j; dr)\nonumber\\
&=\sum^K_{j=1}\int_D(b\phi)(y,j)dy
 \int_{Le^{\lambda_1t_0}}^\infty
\left(\frac{1}{\lambda_1}(\ln r-\ln L)-t_0\right)
r F^{\pi(\phi)}(y,j; dr)\nonumber\\
 &\ge c\sum^K_{j=1}\int_D(b\phi)(y,j)l(y,j)dy=\infty,
\end{align}
 where $c$ is a positive constant. By \cite[Exercise 1.3.8]{D2},
\begin{equation}\label{Durrett-domi}
\Pi_{\phi\mu}^\phi\left(\xi_T\geq \frac{1}{2}\Pi_{\phi\mu}^\phi
\xi_T\right)\geq \frac{(\Pi_{\phi\mu}^\phi
\xi_T)^2}{4\Pi_{\phi\mu}^\phi(\xi_T^2)}.
\end{equation}
 If we can prove that there is a constant $\widehat c>0$ such that
 or all $T>t_0$,
\begin{eqnarray}\label{uniform lower bound}
\frac{(\Pi_{\phi\mu}^\phi
\xi_T)^2}{4\Pi_{\phi\mu}^\phi(\xi_T^2)}\geq
 \widehat c.
\end{eqnarray}
Then by \eqref{Durrett-domi}  we would get
$$
\Pi_{\phi\mu}^\phi\left(\xi_T\geq \frac{1}{2}\Pi_{\phi\mu}^\phi
\xi_T\right)  \geq \widehat c,
$$
and therefore
\begin{eqnarray*}
&&\Pi^\phi_{\phi\mu}\left(\xi_\infty\geq
\frac{1}{2}\Pi_{\phi\mu}^\phi \xi_T\right)
\geq\Pi^\phi_{\phi\mu}\left(\xi_T\geq \frac{1}{2} \Pi_{\phi\mu}^\phi
\xi_T\right)  \ge\widehat c>0.
\end{eqnarray*}
Since $\lim_{T\rightarrow\infty}\Pi_{\phi\mu}^\phi \xi_T=\infty$
(see \eqref{mean-limit}), the above inequality implies
\eqref{positive-prob}.  Now we only need to prove
 \eqref{uniform lower bound}. For this purpose we first estimate
$\Pi_{\phi\mu}^\phi(\xi_T^2)$:
\begin{eqnarray*}
\Pi_{\phi\mu}^\phi
\xi_T^2&=&\Pi_{\phi\mu}^\phi\int_0^Tdt  \int_{L\pi(X_t,  Y_t; \phi)^{-1}
\mbox{e}^{\lambda_1t}}^\infty
\frac{b\pi(\phi)}{\phi}(X_t, {Y}_t)rF(X_t, {Y}_t; dr)\\
&&\quad\times\int_0^Tds
\int_{L\pi(X_s,  Y_s;\phi)^{-1}\mbox{e}^{\lambda_1s}}^\infty
\frac{b\pi(\phi)}{\phi}(X_s, {Y}_s)uF(X_s, {Y}_s; du)\\
&=&2\Pi_{\phi\mu}^\phi\int_0^Tdt
\int_{L\pi(X_t,  Y_t; \phi)^{-1}
\mbox{e}^{\lambda_1t}}^\infty\frac{b\pi(\phi)}{\phi}(X_t, {Y}_t)
r F(X_t, {Y}_t; dr)\\
&&\quad\times \int_t^Tds
\int_{L\pi(X_s,  Y_s; \phi)^{-1}\mbox{e}^{\lambda_1s}}^\infty
\frac{b\pi(\phi)}{\phi}(\widehat X_s,\widehat{Y}_s)u F(X_s, {Y}_s; du)\\
&=&2\Pi_{\phi\mu}^\phi\int_0^Tdt
\int_{L\pi(X_t, Y_t; \phi)^{-1}\mbox{e}^{\lambda_1t}}^\infty
\frac{b\pi(\phi)}{\phi}(X_t, {Y}_t)r F(X_t, {Y}_t; dr)\\
&&\quad\times
\int_t^{(t+t_0)\wedge T}ds\int_{L\pi(X_s, Y_s; \phi)^{-1}
\mbox{e}^{\lambda_1s}}^\infty\frac{b\pi(\phi)}{\phi}(X_s, {Y}_s) uF(X_s,  {Y}_s; du)\\
&&+2\Pi_{\phi\mu}^\phi\int_0^{ T}dt
\int_{L\pi(X_t,  Y_t; \phi)^{-1}\mbox{e}^{\lambda_1t}}^\infty
\frac{b\pi(\phi)}{\phi}(X_t, {Y}_t) r F(X_t, {Y}_t; dr)\\
&&\quad\times
\int_{(t+t_0)\wedge T}^Tds\int_{L\pi(X_s,  Y_s; \phi)^{-1}
\mbox{e}^{\lambda_1s}}^\infty \frac{b\pi(\phi)}{\phi}(X_s, {Y}_s)u F(X_s, {Y}_s; du)\\
&=&III+IV,
\end{eqnarray*}
where
\begin{eqnarray*}
III&=&\displaystyle 2\Pi_{\phi\mu}^\phi\int_0^Tdt
\int_{L\pi(X_t,  Y_t; \phi)^{-1} \mbox{e}^{\lambda_1t}}^\infty
\frac{b\pi(\phi)}{\phi}(X_t, {Y}_t)r F(X_t,{Y}_t; dr)\\
&&\quad\times
\int_t^{(t+t_0)\wedge T}ds\int_{L\pi(X_s,  Y_s; \phi)^{-1}
\mbox{e}^{\lambda_1s}}^\infty \frac{b\pi(\phi)}{\phi}(X_s, {Y}_s)u F(X_s, {Y}_s;du)
\end{eqnarray*}
and
\begin{align*}
IV= & 2\Pi_{\phi\mu}^\phi\int_0^{T}dt
\int_{L\pi(X_t,  Y_t; \phi)^{-1}\mbox{e}^{\lambda_1t}}^\infty
\frac{b\pi(\phi)}{\phi}(X_t, {Y}_t)rF(X_t, {Y}_t; dr)\\
&\quad\times
\int_{(t+t_0)\wedge T}^Tds\int_{L\pi(X_s, Y_s; \phi)^{-1}
\mbox{e}^{\lambda_1s}}^\infty\frac{b\pi(\phi)}{\phi}(X_s, {Y}_s) u F(X_s, {Y}_s; du)\\
=& 2\sum^K_{j=1}\int_0^{T}dt\int_D(b\phi) (y,j)dy
\int_{L\pi(y,j;\phi)^{-1}\mbox{e}^{\lambda_1t}}^\infty
r\pi(y,j;\phi) F(y,j; dr)\\
& \ \ \ \times
\int_{(t+t_0)\wedge T}^Tds
\int_Dp^{\phi}(s-t,(y, j),(z,k))\frac{b\pi(\phi)}{\phi}(z,k)dz
\int_{L\pi(z,k;\phi)^{-1}\mbox{e}^{\lambda_1s}}^\infty u F(z,k;du).
\end{align*}
By our assumption we have $ \|\int_1^\infty
r\,F(\cdot;dr)\|_{\infty}<\infty$. Since
$L\inf_{(x,i)\in D\times S}\phi(x,j)^{-1}\ge 1$,  we have
$$
III\leq c_1\Pi^{\phi}_{\phi\mu}\xi_T,
$$
for some positive constant $ c_1$ which does not depend on $T$.
Using \eqref{domi-p} and the definition of $n^{\pi(\phi)}$, we get that
$$
\begin{array}{rl}
&\displaystyle
\int_{(t+t_0)\wedge T}^Tds
\int_Dp^{\phi}(s-t,(y,j),(z,k))\frac{b\pi(\phi)}{\phi}(z,k)dz
\int_{L\pi(z,k;\phi)^{-1}\mbox{e}^{\lambda_1s}}^\infty
u F(z,k; du)\\
\le&\displaystyle
2\int_{(t+t_0)\wedge T}^Tds\int_D(b\phi)(z,k)dz\int_{L\phi(z,k;\phi)^{-1}
\mbox{e}^{\lambda_1s}}^\infty \pi(z,k;\phi)u F(z,k; du)\\
\le&\displaystyle
2\int_{t_0}^Tds\int_D(b\phi)(z,k)dz\int_{Le^{\lambda_1s}}^\infty
rF^{\pi(\phi)}(z,k; dr)\\
=&\displaystyle2\sum^K_{k=1}
\int_{t_0}^Tds\int_D(b\phi)(z,k)dz\int_{Le^{\lambda_1s}}^\infty
rF^{\pi(\phi)}(z,k; dr)=2A_T.\end{array}
$$
Then using \eqref{domi-mean},  we have
$$
IV\leq 4 A_T\Pi^{\phi}_{\phi\mu}\xi_T\leq 4 (\Pi^{\phi}_{\phi\mu}
\xi_T)^2.
$$
Combining the estimates above on $III$ and $IV$, we get that there
exists a $c_2>0$ independent of $T$ such that for $T>t_0$,
$$
\Pi_{\phi\mu}^\phi(\xi_T^2)\le 4
(\Pi_{\phi\mu}^\phi(\xi_T))^2+
c_1\Pi_{\phi\mu}^\phi(\xi_T)\le
c_2(\Pi_{\phi\mu}^\phi(\xi_T))^2.
$$
Then we have \eqref{uniform lower bound} with $\widehat c=1/c_2$,
and the proof of the theorem is now complete.\qed
\end{proof}

\begin{definition}\label{con-mart}
Suppose that $(\Omega, {\cal F}, \bP)$ is a probability space,
$\{{\cal F}_t,t\ge 0\}$ is a filtration on $(\Omega, {\cal F})$ and
that ${\cal G}$ is a sub-$\sigma$-field of ${\cal F}$. A real valued
process $U_t$ on $(\Omega, {\cal F}, \bP)$ is called a $\bP(\cdot |\
{\cal G})$-martingale (submartingale, supermartingale resp.) with
respect to $\{{\cal F}_t,t\ge 0\}$ if (i) it is adapted to $\{{\cal
F}_t\vee{\cal G},t\ge 0\}$; (ii) for any $t\geq 0,\
\bE(|U_t|\big|{\cal G})<\infty$ and (iii) for any $t>s$,
$$
\bE(U_t\big|{\cal F}_s\vee{\cal G})=(\ge,\le\mbox{ resp. })\
U_s,\quad{\rm a.s.}
$$
\end{definition}

We need the following  result. For its proof, see \cite[Lemma 3.3]{LRS09}.

\begin{lemma}\label{con-mart-conver}
Suppose that $(\Omega, {\cal F}, \bP)$ is a probability space,
$\{{\cal F}_t,t\ge 0\}$ is a filtration on $(\Omega, {\cal F})$ and
that ${\cal G}$ is a $\sigma$-field of ${\cal F}$. If $U_t$  is a
$\bP(\cdot|\ {\cal G})$-submartingale with respect to $\{{\cal
F}_t,t\ge 0\}$ satisfying
\begin{equation}\label{f-con-exp}
\sup_{t\ge 0} \bE(|U_t|\big|{\cal G})<\infty\quad {\rm a.s.}
\end{equation}
then there exists a finite random variable $U_{\infty}$ such that
$U_t$ converges a.s. to $U_{\infty}$.
\end{lemma}

We are now in the position to prove the main result of this paper.

\medskip

\noindent {\bf Proof of Theorem $\ref{maintheorem}$.}\quad

Recall that, by Proposition \ref{equi-deg}, to prove Theorem
\ref{maintheorem}, we only need to consider the case
$d\mu={\phi}(x,i)dxdi$, where $di$ is the counting measure on $S$.

 We first prove that
if $\sum^K_{i=1}\int_{D}{\phi}(x,i)b(x,i)l(x,i)dx<\infty $, then $W_{\infty}$ is
non-degenerate under $\bP_{\mu}$.
Since $W_{t}(\phi)$ is a nonnegative martingale,
to show it is a closed martingale,
it suffices to prove
$\bP_{\mu}(W_\infty(\phi))=\bP_{\mu}(W_0 (\phi))=  \langle\phi, \mu\rangle$.
Since $W_t^{-1}(\phi)$ is a positive
supermartingale under $\widetilde{\bP}_\mu$, $W_t(\phi)$ converges to
some nonnegative random variable $W_\infty(\phi)\in (0,\ \infty]$
under $\widetilde{\bP}_\mu$.
By \cite[Theorem 5.3.3]{D2}, we only need to
prove that
\begin{equation}\label{to-prove}
\widetilde{\bP}_\mu\left(W_\infty(\phi)<\infty\right)
=1.\end{equation}
By \eqref{decomp},  $(\chi_t, t\ge 0;\widetilde{\bP}_\mu)$ has the same
law as $(\chi_t+\hat \chi_t, t\ge 0; \bP_{\mu,\phi})$, where $\{\chi_t, t\ge 0;\bP_{\mu,\phi})$ is a copy of
$(\chi_t, t\ge 0; \bP_{\mu})$, and $\hat \chi_t=\sum_{s\in(0,t]\cap D_J}\chi^s_t.$ Put
\begin{eqnarray}\label{definition of W_t}
M_t(\phi):=\sum_{s\in(0,t]\cap{\mathcal D_J}}\langle \phi,
\chi^s_t\rangle\ \mbox{e}^{-\lambda_1t}.
\end{eqnarray}
Then
\begin{equation}\label{decomp3}
(W_t(\phi), t\ge 0; \widetilde \bP_{\mu})
= (W_t(\phi)+M_t(\phi), t\ge 0; \bP_{\mu,\phi})\quad \mbox{ in law},
\end{equation}
where $\{W_t(\phi), t\ge 0\}$ is copy of the martingale defined in
\eqref{martingale-1} and is independent of $M_t(\phi)$. Let ${\cal
G}$ be the $\sigma$-field generated by $\{Y_t, m_t,
t\ge 0\}$.  Then, conditional on ${\cal G}$, $(\chi^{s}_t, t\ge
s, \bP_{\mu,\phi})$
has the same law
as $(\chi_{t-s}, t\ge s, \bP_{m_{s}\delta_{\widehat{Y}_{s}}} )$  and
$(\chi^{s}_t, t\ge s, \bP_{\mu,\phi})$ are independent for
$s\in {\mathcal D}_J$. Then we have
\begin{eqnarray}\label{form2-W_t}
M_t(\phi)
 \stackrel{d}{=}\sum_{s\in(0,t]\cap {\mathcal D}_J}
 \mbox{e}^{-\lambda_1 s}
 W_{t-s}^{s}(\phi),
 \end{eqnarray}
where for each $s\in{ \mathcal D_J}$, $W^{s}_t(\phi)$ is
a copy of the martingale defined by \eqref{martingale-1} with
$\mu=m_{s}\delta_{\widehat{Y}_{s}}$, and conditional
on ${\cal G}$, $\{W^{s}_t(\phi),t\ge 0\}$
are independent for $s\in {\mathcal D}_J$. To prove \eqref{to-prove}, by \eqref{decomp3}, it
suffices to show that
$$
\bP_{\mu,\phi}\left(\lim_{t\to\infty}[W_t(\phi)+M_t(\phi)]<\infty\right)
=1.
$$
Since $(W_t(\phi), t\ge 0)$ is a nonnegative martingale under the
probability $\bP_{\mu,\phi}$, it converges $\bP_{\mu,\phi}$ almost
surely to a finite random variable $W_\infty(\phi)$ as $t\to\infty$.
So we only need to prove
\begin{equation}\label{M-finite}
{\bP}_{\mu,\phi}\big(\lim_{t\to\infty}M_t(\phi)<\infty\big)
=1.
\end{equation}
Define ${\cal H}_t:={\cal G}\bigvee\sigma( \chi^{\sigma}_{(s-\sigma)}; \sigma\in[0,t]\cap {\mathcal D_m}, s\in[\sigma, t]).$ Then
$(M_t(\phi))$ is a $\bP_{\mu,\phi}(\cdot|{\cal G})$-nonnegative
submartingale with respect to $({\cal H}_t)$.
By
\eqref{form2-W_t}  and Lemma \ref{lemma1},
\begin{align*}
\sup_{t\geq 0}\bP_{\mu,\phi}\big(M_t(\phi)|{\cal G}\big)&=
\sup_{t\geq 0} \sum_{{s\in [0,\ t]\cap\mathcal
D_J}}\mbox{e}^{-\lambda_1s}m_s\phi({\widehat X}_{s},{\widehat Y}_{s})\\
& \leq
\sum_{{s\in\mathcal
D_J}}\mbox{e}^{-\lambda_1s}m_s\phi({\widehat X}_{s},{\widehat Y}_{s})<\infty,\
\bP_{\mu, \phi}\mbox{-a.s.}
\end{align*}
Then by Lemma \ref{con-mart-conver}, $M_{t}(\phi)$ converges
$\bP_{\mu,\phi}$-a.s. to  $M_{\infty}(\phi)$   as $t\to\infty$ and
$\bP_{\mu,\phi}(M_{\infty}(\phi)<\infty)=1$,
which establishes \eqref{M-finite}.

Now we prove the other direction.  Assume that
$\sum^K_{i=1}\int_D{\phi}(y,i)b(y,i){l(y,i)}dy=\infty$.
We are going to prove that $W_\infty(\phi):=\lim_{t\to\infty}W_t(\phi)$ is degenerate
with respect to $ \bP_{\mu}$.
By \cite[Proposition 2]{HR}, $\frac{1}{W_t(\phi)}$  is a
supermartingale under $\widetilde \bP_{\mu}$, and thus
$1/[M_{t}(\phi)+W_t(\phi)]$ is a nonnegative supermartingales under
$\bP_{\mu,\phi}$.  Recall that $W_{t}(\phi)$ is a nonnegative martingale under $\bP_{\mu,\phi}$. Then the limits
$\lim_{t\to\infty}W_t(\phi)$ and
$1/\lim_{t\to\infty}[M_{t}(\phi)+W_t(\phi)]$ exist and finite
$\bP_{\mu,\phi}$-a.s. Therefore $\lim_{t\to\infty}M_t(\phi)$ exists in
$[0, \infty]$ $\bP_{\mu,\phi}$-a.s. Recall the definition of $(\eta_i,
\sigma_i; i=1,2,\cdots)$ in Lemma \ref{lemma1}, and note that
$\lim_{i\to\infty}\sigma_i=\infty$. By Lemma \ref{lemma1},
$$
\limsup_{t\rightarrow\infty}M_t(\phi)\geq\limsup_{i\rightarrow\infty
}M_{\sigma_i}(\phi) \geq\limsup_{i\rightarrow\infty}
\mbox{e}^{-\lambda_1\sigma_i}\eta_i\phi(\widehat X_{\sigma_i},
\widehat Y_{\sigma_i})=\infty\quad \bP_{\mu,\phi}\mbox{-a.s.}
$$
So we have
$$
\lim_{t\rightarrow\infty}M_t(\phi)=\infty \quad \bP_{\mu,\phi}\mbox{-a.s.}
$$
By \eqref{decomp3},
$$
\widetilde{\bP}_{\mu}(W_{\infty}(\phi)=\infty)=1.
$$
It follows from  \cite[Theorem 5.3.3]{D2} that
$\bP_{\mu}(W_{\infty}=0)=1$.\qed

\section{Appendix: Non-local Feynman-Kac transform}

In this Appendix, we establish a result on time-dependent non-local Feynman-Kac transform, which has been used in the proof of Theorem \ref{theorem1}.

\medskip

Let $E$ be a Lusin space and ${\cal B}(E)$ be the Borel $\sigma$-field on $E$,
and let $m$ be a $\sigma$-finite measure on ${\cal B}(E)$ with ${\rm supp} [m] = E$.
Let $\{\xi_t, t\ge  0; \Pi_{x}\}$  be
an $m$-symmetric Borel standard process on $E$
with L\'{e}vy system  $(J, t)$,
 where $J(x, dy)$ is a kernel from $(E, {\cal B}(E))$ to $(E\cup\{\partial\}, {\cal B}(E\cup\{\partial\}))$.

\begin{lemma}\label{G-nonlocal-FK}
Suppose that $\{\xi_t, t\ge  0; \Pi_{x}\}$ is
an $m$-symmetric Borel standard process on $E$ with
 L\'{e}vy system  $(J, t)$.
Assume that $q$ is a locally bounded function on $[0,\infty)\times E$
and that $F$ is a non-positive,  ${\cal B}([0,\infty)\times E\times E)$-measurable
 function vanishing on the diagonal  of $E\times E$
so that
for any $x\in E$,
\begin{equation}\label{e:6.1}
\sum_{0< s\le t}F(t-s,\, \xi_{s-},\, \xi_s)   > -\infty \quad \hbox{for every } t>0
\quad \Pi_x\mbox{-a.s.}
\end{equation} and
\begin{equation}\label{e:6.2}
\sup_{x\in E} \Pi_x \left[ \int_0^t \int_{E_\partial}(1-e^{F(t-s,\,\xi_{s},\, y)})J(\xi_s, dy) ds \right]
<\infty \quad \hbox{for every } t>0.
\end{equation}
 For any $x\in E$, $t\ge 0$ and $f\in B^+_b(E)$, define
\begin{equation}\label{non-local-FK}
h(t,x):=\Pi_{x}\left[e^{\int^t_0q(t-s, \xi_s)ds+\sum_{0< s\le t}F(t-s,\, \xi_{s-},\, \xi_s)}f(\xi_t)\right].\end{equation}
Then $h$ is the unique locally bounded positive solution of the following integral equation
\begin{align}\label{inteq-nonlocal-FK}h(t,x)=&\Pi_xf(\xi_t)+\Pi_x\int^t_0q(t-s,\xi_s)h(t-s, \xi_s)ds\nonumber\\
&+\Pi_{x}\left[\int^t_0\int_E(e^{F(t-s,\,\xi_{s},\, y)}-1)h(t-s, y)J(\xi_s, dy)ds\right].
\end{align}
\end{lemma}

\begin{proof}
Note that under the locally boundedness assumption of $q(t, x)$ and \eqref{e:6.1}, the function $h$
of \eqref{non-local-FK} is
 well defined
and positive, and there exists $c>0$ such that
$$
h(t, x)\le e^{ct}\Pi_x[f(\xi_t)].
$$
Thus $h(t,x)$ is bounded on $[0,T]\times E$ for any $T>0$.
The assumption \eqref{e:6.2} implies that the last term  of \eqref{inteq-nonlocal-FK}
is absolutely convergent and defines a bounded function on $[0, T]\times E$ for every $T>0$.
For $s\le t$,  define
$$
A_{s,t}=\int^t_sq(t-r, \xi_r)dr+\sum_{s<r\le t}F(t-r, \xi_{r-},\xi_r),
$$
which is right continuous and has left limits as a function of $s$. Note that
\begin{eqnarray*}
e^{A_{0,t}}-1
&=& - \left( e^{A_{t,t}}-e^{A_{0,t}} \right) \\
&=&\int^t_0e^{A_{s-}, t}q(t-s, \xi_s)ds-\sum_{0<s\le t} \left( e^{A_{s,t}}-e^{A_{s-, t}} \right)\\
&=&\int^t_0 e^{A_{s, t}} q(t-s,\xi_s)ds+\sum_{0<s\le t} e^{A_{s,t}} \left( e^{F(t-s,\,\xi_{s-},\, \xi_s))}-1 \right).
\end{eqnarray*}
Hence we have
\begin{align*}
&\Pi_{x}\left[\left(e^{A_{0,t}}-1\right)f(\xi_t)\right]\\
=&\Pi_x\left[\int^t_0e^{A_{s, t}}q(t-s,\xi_s)f(\xi_t)ds\right]+\Pi_{x}\left[\sum_{0<s\le t}
e^{A_{s,t}}
\left( e^{F(t-s,\,\xi_{s-},\, \xi_s))}-1 \right) f(\xi_t)\right].
\end{align*}
By the Markov property of $\xi$ and the fact that
\begin{eqnarray*}
A_{s,t} =
\left( \int^{t-s}_0 q(t-s-r,\xi_r)dr+
\sum_{0< r\le t-s}
F(t-s-r,\, \xi_{r-},\, \xi_r ) \right)
\circ\theta_s ,
\end{eqnarray*}
we have
\begin{align*}
& h(t,x)\\
&= \Pi_xf(\xi_t)+\Pi_x\left[\int^t_0q(t-s,\xi_s)\Pi_{\xi_s}\left(e^{\int^{t-s}_0q(t-s-r,\xi_r)dr+
\sum_{0< r\le t-s}
F(t-s-r,\, \xi_{r-},\, \xi_r)}f(\xi_{t-s})\right)\right]\\
&\quad+\Pi_{x}\left[\sum_{0<s\le t} \left(e^{F(t-s,\,\xi_{s-},\, \xi_s))}-1 \right)
\Pi_{\xi_s}\left[e^{\int^{t-s}_0q(t-s-r, \xi_r)dr+
\sum_{0< r\le t-s}
F(t-s-r,\, \xi_{r-},\, \xi_r)}f(\xi_{t-s})\right]\right].
\\
&=\Pi_xf(\xi_t)+\Pi_x\int^t_0q(t-s,\xi_s)h(t-s, \xi_s)ds+\Pi_{x}\left[\sum_{0<s\le t}
\left( e^{F(t-s,\,\xi_{s-},\, \xi_s))}-1 \right) h(t-s, \xi_s)\right]
\\
&=\Pi_xf(\xi_t)+\Pi_x\int^t_0
q(t-s, \xi_s)
h(t-s, \xi_s)ds\\
&\quad+\Pi_{x}\left[\int^t_0\int_E \left( e^{F(t-s,\,\xi_{s},\, y)}-1 \right) h(t-s, z)
J(\xi_s, dy) ds\right].
\end{align*}
Thus $h(t,x)$ defined by \eqref{non-local-FK} is a locally bounded positive solution of \eqref{inteq-nonlocal-FK}.

It follows from \cite[Proposition 2.15]{Li} that
\eqref{inteq-nonlocal-FK} has a unique locally bounded positive  solution.
\qed
\end{proof}

\begin{remark}\label{rem6.2} \rm
\begin{description}
\item{(i)}  Lemma \ref{G-nonlocal-FK} can be easily extended to signed $F$ (with the same argument)
by replacing condition \eqref{e:6.1}-\eqref{e:6.2} by
$$
\sum_{0< s\le t}F^-(t-s,\, \xi_{s-},\, \xi_s)  < \infty \quad \hbox{for every } t>0
\quad \Pi_x\mbox{-a.s.} \eqno (6.1')
$$
 and
$$
\sup_{x\in E} \Pi_x \left[ \int_0^t \int_{E_\partial} \left| 1-e^{F(t-s,\,\xi_{s},\, y)} \right| (\xi_s, dy) ds \right]
<\infty \quad \hbox{for every } t>0.  \eqno (6.2')
$$

\medskip

\item{(ii)} If $F$ does not depend on $t$, the above result
follows easily from the results of \cite{CS03a}.

\item{(iii)} If $\sup_{x\in E}J(x, E\cup\{\partial\})<\infty$, or if
\begin{equation}
\sup_{x\in E} \Pi_x \left[ \int_0^t \int_{E_\partial} |F(t-s, \xi_s, y)| J(\xi_s, dy) ds \right]
<\infty \quad \hbox{for every } t>0,
\end{equation}
then conditions \eqref{e:6.1} and \eqref{e:6.2} are satisfied.
\end{description}
\end{remark}

\bigskip

\begin{singlespace}

\end{singlespace}
\end{doublespace}
\vskip 0.3truein

 {\bf Zhen-Qing Chen}

Department of Mathematics, University of Washington, Seattle,
WA 98195, USA

E-mail: zqchen@uw.edu

\bigskip
{\bf Yan-Xia Ren}

LMAM School of Mathematical Sciences \& Center for
Statistical Science,
Peking
University,
Beijing, 100871, P.R. China.

E-mail: yxren@math.pku.edu.cn
\bigskip

 {\bf Renming Song:}

  Department of Mathematics,
 The University of Illinois,  Urbana, IL 61801 U.S.A.,

  E-mail: {\tt rsong@illinois.edu}
 \bigskip

\small
\begin{thebibliography}{99}
\bibitem{AH76a} S. Asmussen and H. Hering (1976):
Strong limit theorems for general supercritical branching processes
with applications to branching diffusions.
{\em Z. Wahrscheinlichkeitstheor. verw. Geb.}, {\bf 36}, 195-212.

\bibitem{Ban} R. Banuelos (1991):
Intrinsic ultracontractivity and eigenfunction estimates for Schr\"{o}dinger operators. {\em J. Funct. Anal.}, {\bf 100}, 181--206.

\bibitem{BK} J. D. Biggins and A.E. Kyprianou (2004): Measure
change in multitype branching. {\em Adv. in Appl. Probab.}, {\bf
36}, 544-581.

\bibitem{CS03a} Z.-Q. Chen and R. Song (2003): Conditional gauge theorem
for non-local Feynman-Kac transforms. {\em Probab. Theory Relat. Fields  \bf 125},
45-72.




\bibitem{CZ} Z.-Q. Chen and Z. Zhao (1996): Potential theory for elliptical systems.
{\em Ann. Probab.}, {\bf 24}, 293-319.


\bibitem{CKP} S. Cho, P. Kim and H. Park (2012):
Two-sided estimates on Dirichlet heat kernels for time-dependent parabolic operators
with singular drifts in $C^{1, \alpha}$-domains.
{\it J. Diff. Equations \bf 252}, 1101-1145.


\bibitem{DS} E. B. Davies and B. Simon (1984):
Ultracontractivity and heat kernels for Schr\"{o}dinger operators and Dirichlet Laplacians. {\em J. Funct. Anal.}, {\bf 59}, 335-395.

\bibitem{DGL} D. A. Dawson, L. G. Gorostiza and Z. Li (2002):
Nonlocal branching superprocesses and some related models. {\em Acta Appl. Math.}, {\bf 74}, 93-112.



\bibitem{D2}  R. Durrett (2010): {\it Probability: theory and Examples (4th edition)}. Cambridge University Press, Cambridge.

\bibitem{Dy} E. B. Dynkin (1993):  Superprocesses and partial
differential equations. {\em Ann. Probab.}, {\bf 29}, 1833-1858.


\bibitem{DKS} E. B. Dynkin, S. E. Kuznetsov and  A. V. Skorokhod (1994):
Branching measure-valued processes. {\em  Probab. Theory Related Fields}, {\bf 99}, 55-96.


\bibitem{EK}  J. Englander and A. E. Kyprianou (2004):
Local extinction versus local exponential growth for spatial
branching processes.  {\em Ann. Probab.}, {\bf 32}, 78-99.

\bibitem{E}  S. N. Evans (1992):  Two representations of a
conditioned superprocess. {\em Proc. Roy. Soc. Edinb.}, {\bf 123A},
959-971.


\bibitem{F}  M. Freidlin (1985):  {\it Functional Integration and Partial Differential Equations}.  Princeton Univ. Press.



\bibitem{HR} S. C. Harris and M. I. Robert (2009): Measure changes
with extinction. \emph{Statit. Probab. Lett.} \textbf{79}, 1129--1133.

\bibitem{KS} H. Kesten and B. P. Stigum (1966): A limit theorem
for multidimensional Galton-Watson process. {\it Ann. Math.
Statist.}, {\bf 37}, 1211-1223.

\bibitem{KLMR} A. E. Kyprianou, R.-L. Liu,  A. Murillo-Salas and Y.-X. Ren (2012): Supercritical super-Brownian motion
with a general branching mechanism and travelling waves. \textit{Ann. Inst. Henri Poincar\'{e} Probab. Stat.},
{\bf 48}, 661-687.

\bibitem{KM} A. E. Kyprianou and A. Murillo-Salas (2013): Super-Brownian motion: $L^{p}$ convergence of martingales through the pathwise spine decomposition.
In {\em Adavances in Superprocesses and Nonlinear PDEs}, volume 38 of {\em Springer Proceedings in Mathematics and Statistics}.

\bibitem{KP}  A. E. Kyprianou and S. Palau (2016): Extinction properties of multi-type continuous-state branching processes. Preprint,
https://arxiv.org/abs/1604.04129v2.


\bibitem{KS1} P. Kim and R. Song (2008): Intrinsic untracontractivity
of non-symmetric diffusion semigroups in bounded domains. {\em
Tohoku Math. J.}, {\bf 60}, 527--547.



\bibitem{KLPP} T. G. Kurtz, R. Lyons, R. Pemantle and Y. Peres
(1997): A conceptual proof of the Kesten-Sigum theorem for multitype
branching processes. In {\it Classical and Modern Branching
processes} (K. B. Athreya and P. Jagers, eds), {\bf 84},  181-186,
Springer-Verlag,  New York.


\bibitem{Li} Z. Li (2011): \textit{Measure-valued Branching Markov Processes}. Springer, Heidelberg.

\bibitem{LRS09} R. Liu, Y.-X. Ren and R. Song (2009): $L\log L$ criterion for a class of super-diffusions.  {\em J. Appl. Probab.}, {\bf 46}, 479-496.




\bibitem{LPP} R. Lyons, R. Pemantle and Y. Peres (1995):
Conceptual proofs of $L\log L$ criteria for mean behavior of
branching processes. {\em Ann. Probab.}, {\bf 23}, 1125-1138.

\bibitem{PR} Z. Palmowski, and T. Rolski (2002):
A technique for exponential change of measure for Markov processes. {\em Bernoulli} 8(6), 767-785.

\bibitem{RSY} Y.-X. Ren, R. Song and T. Yang (2016): Spine decomposition and $L\log L$ criterion  for  superprocesses with non-local branching mechanisms. Preprint, 2016. arxiv: 1609.02257v1.


\bibitem{Sharpe} M. Sharpe (1988): {\it General Theory of Markov
Processes}. Academic Press, San Diego.

\end{thebibliography}
\end{document}